\documentclass{article}
\usepackage{amsmath,amsfonts,amssymb,amsthm}
\usepackage[french, british]{babel}
\usepackage{color}
\usepackage{hyperref}
\usepackage{dsfont}

\newcommand{\F}{\mathcal{F}}
\newcommand{\calE}{\mathcal{E}}

\newcommand{\calB}{\mathcal{B}}
\newcommand{\calS}{\mathcal{S}}
\newcommand{\calJ}{\mathcal{J}}
\newcommand{\pS}{\partial \mathcal{S}}
\newcommand{\R}{\mathbb{R}}
\newcommand{\calR}{\mathcal{R}}

\newcommand{\V}{\mathcal{V}}
\newcommand{\W}{\mathcal{W}}
\newcommand{\bbS}{\mathbb{S}}
\newcommand{\bbO}{\mathbb{O}}
\newcommand{\bbQ}{\mathbb{Q}}
\newcommand{\DIV}{\textnormal{div}\,}
\newcommand{\eps}{\varepsilon}

\def\longrightharpoonup{\relbar\joinrel\rightharpoonup}
\def\cv{\stackrel{w}{\longrightharpoonup}}

\allowdisplaybreaks

\newtheorem{Theorem}{Theorem}
\newtheorem{Definition}{Definition}
\newtheorem{Corollary}{Corollary}
\newtheorem{Proposition}{Proposition}
\newtheorem{Lemma}{Lemma}
\newtheorem{Remark}{Remark}
\begin{document}

\date{\today}
\title{On the vanishing rigid body problem in a viscous compressible fluid}
\author{Marco Bravin \footnote{BCAM – Basque Center for Applied Mathematics, Mazarredo 14, E48009 Bilbao,
Basque Country – Spain}, \v{S}\'arka Ne\v{c}asov\'a\footnote{Institute of Mathematics, Czech Academy of Sciences \v{Z}itn\'a 25, 115 67 Praha 1, Czech Republic.} }

\maketitle

\begin{abstract}
In this paper we study the interaction of a small rigid body in a viscous compressible fluid. The system occupies a bounded three dimensional domain. The object it allowed to freely move and its dynamics follows the Newton's laws. We show that as the size of the object converges to zero the system fluid plus rigid body {converges} to the compressible Navier-Stokes system under some mild lower bound on the mass and the inertia momentum. It is a first result of homogenization in the case of fluid-structure interaction in the compressible situation. As a corollary we slightly improved the result on the influence of a vanishing obstacle in a compressible fluid for $\gamma \geq 6$.
\end{abstract}

\section{Introduction}

In this work we study the interaction of a small rigid body with a compressible viscous fluid. The object is allowed to freely move and its dynamic follows the Newton's laws.

These types of problems have both mathematical and physical interests and have been { investigated}  for fluids with different properties in the last years. The first results in this direction studied the case where the rigid body cannot move. This problem takes the name of obstacle problem and the shrinking limit -called
{homogenization}- has been widely studied in the case of viscous fluids. In  \cite{Tar} the author deduced the Darcy's law from the { homogenization} of the Stokes equations in a particular regime. Later Allaire understood that the { homogenization} process for both Stokes and stationary Navier-Stokes systems depends on the size of the holes and deduce the Darcy's law, the Brinkman's law and the no-influence of the holes in three different settings, see \cite{All} and \cite{All:2}. Later on these results were extended for some regimes to the case of unsteady fluid, see \cite{M, FNN, ILL:visc}. 
Moreover, in \cite{ILL}, \cite{Lac}  where the articles deal with an inviscid, incompressible fluid and the last one with a viscous one. 

These results were then extended to the case of a rigid body which is allowed to move. In particular in \cite{GLS} and \cite{G:S}, they study the problem in the case the fluid is two dimensional, non-viscous and incompressible and in \cite{La:Ta}, \cite{Je:1} and \cite{Je:2} where they consider the case of viscous incompressible fluid in both two and three dimensions.

For compressible viscous fluid, again when a rigid body is not moving,  the homogenization for perforate domains was studied in  \cite{Mas}, \cite {FT}, \cite{Fei:Lu}  and \cite{Lu:Schw}. In the first two papers the size of the holes and the mutual distance were comparable and a Darcy's law was derived. Recently this result was extended in \cite{KHS}. In the last two works the authors consider the case of tiny holes and they recover respectively the stationary and non-stationary compressible Navier Stokes system.  An extension of the homogenization for perforate domain in the case of steady full compressible system where authors again recover the full system, see \cite{LP}. Let us mention that homogenization problem in the compressible case was proved only in three dimensional case. 

%{\color{blue}Homogenization} for viscous fluid has been widly studied in the past years. In  \cite{Tar} the author deduced the Darcy's law from the {\color{blue} homogenization} of the Stokes equations in a particular regime. Later Allaire understood that the {\color{blue} homogenization} process for both Stokes and stationary Navier-Stokes systems depends on the size of the holes and deduce the Darcy's law, the Brinkman's law and the no-influence of the holes in three different settings, see \cite{All} and \cite{All:2}. Later on these results were extended for some regimes to the case of unsteady fluid, see \cite{M, FNN}. 

In this paper we consider the case of a shrinking object in a compressible viscous { barotropic} fluid with pressure $ p(\rho) = a \rho^{\gamma} $, { $a$ is a positive constant, for simplicity we consider $a=1$,  an assumption on the constant $\gamma$ will be precised later.}  In particular we show that in the limit the presence of the small object does not influence the dynamics. This is a first result in the case of moving shrinking object in the compressible barotropic fluid.

The idea of the proof is to use a cut off  { function} around the rigid body and to pass to the limit in the weak { formulation} of the equations. In three dimension the gradient of the cut-off is bounded in $ L^p $ for $ 1 \leq p \leq 3 $, so we need that the pressure is $ L^p $, for $ p \geq 3\gamma/2 $ which gives the restriction $ \gamma \geq 6 $, to pass to the limit in the pressure term. { Let us mention that the existence of weak solution is valid for $\gamma >3/2$.}  

We impose a second restriction on the mass and the { inertia} matrix of the object of size $ \varepsilon $. This condition allows us to have an a priori estimate of the velocity of the solid from the energy estimate. The condition reads
\begin{equation*}
m_{\eps} \eps^{-\frac{3\gamma -4}{\gamma} }  \longrightarrow \infty  \quad \text{ and } \quad \left(\inf_{\xi \in S^2} \xi \cdot \calJ_{\calS_{0,\eps}} \xi \right)\eps^{-2 - \frac{3\gamma -4}{\gamma} } \longrightarrow + \infty \quad \text{ as } \eps \longrightarrow 0.
\end{equation*}

This condition will naturally appear when we prove the improved regularity of the pressure which is the main difficulty and difference with respect to the case of incompressible fluids  where {\it no restrictions are need on the mass and the inertia matrix of the object}.

In the next section we introduce the problem at a mathematical level.

\subsection{Formulation of problem}

In this paper we study the interaction between a compressible viscous fluid and a small rigid body which is allowed to freely move. In particular we will show that as the size of the object tends to zero the presence of the rigid body is negligible. Let start by recalling the equations satisfied by the fluid plus rigid body system. 

\begin{align}
\partial_t \rho_{\F} + \DIV( \rho_{\F}u_{\F}) & \, = 0 \quad \quad  \, \, \, \, \text{ for } x  \in \F(t),  \nonumber \\
\partial_t( \rho_{\F} u_{\F}) + \DIV(\rho_{\F} u \otimes u_{\F} ) - \DIV\bbS(\nabla u_{\F}) + \nabla p( \rho_{\F}) & \, =  0   \quad \quad \, \, \, \, \text{ for } x \in \F(t), \nonumber \\
u_{\F} & \, =  0 \quad \quad \, \, \, \,  \text{ for } x \in \partial \Omega, \nonumber \\
u_{\F} & \, = u_{\calS} \quad \quad  \text{ for } x \in \pS(t), \nonumber \\
m \ell'(t) = - \oint_{\pS(t)} \left(\bbS(\nabla u_{F}) - p(\rho_{\F})\mathbb{I}  \right)& nds,  &&  \label{equ:CNS:RB} \\
\calJ(t) \partial_t \omega(t) = \calJ(t) \omega(t) \times \omega(t) - \oint_{\pS(t)} (x-h(t))\times &\left(\bbS(\nabla u_{F}) - p(\rho_{\F})\mathbb{I}  \right) n ds, \nonumber \\
(\rho_{\F}u_{\F})(0) = q_{\F,0}, \quad \rho_{\F}(0) & \, = \rho_{F,0}  \quad \quad \, \, \, \,  \text{ for } x \in \F_{0},\nonumber \\
\ell(0) = \ell_0, \quad \omega(0) & \,= \omega_0. \nonumber
\end{align}  
In the above system $ \calS(t) \subset \Omega $ is the position of the solid at time $ t $ and $ \F(t) = \Omega \setminus \calS(t)  $ is the part of the domain occupied by the fluid. The scalar quantity $ \rho_{\F} $ describes the pressure of the fluid. The vector field $ u_{\F} $ is the velocity of the fluid. The $ 3 \times 3 $ matrix 
\begin{equation*}
\mathbb{S}(\nabla u_{\F})-p(\rho_{F})\mathbb{I} = \mu D(u_{\F}) + (\lambda+\mu) \DIV(u_{\F}) \mathbb{I} - \rho_{\F}^{\gamma}\mathbb{I} = \mu \frac{\nabla u_{\F}+\nabla u_{\F}^T}{2} + (\lambda+\mu) \text{tr}(\nabla u_{\F})\mathbb{I} - \rho_{\F}^{\gamma}\mathbb{I} 
\end{equation*}
is the stress tensor where $ \mu > 0 $ is the viscous { coefficient} and $ 3 \lambda + 2 \mu \geq 0 $, { the coefficient} $ \gamma \geq 6 $. The vector $ h(t) $ is the position of the center of mass of the solid, the velocity of the center of mass $ h'(t) = \ell $, the angular velocity is denoted by $ \omega $ and $ u_{\calS} = \ell + \omega \times (x-h(t))$ is the solid velocity. The mass of the solid $ m $, the position of the center of mass $ h(t)$ and  the inertia matrix of the solid $ \calS(t) $ are defined as 
\begin{equation*}
\begin{array}{l}
m = \int_{{\calS}(t)} \rho_{\calS} \ dx, \\
h(t) = \frac{1}{m} \int_{{\calS}(t)} \rho_{{\calS}} \  x \ dx, \\
\calJ(t)= \int_ {{\calS}(t)} \rho_{{\calS}}  \big[ |x-h(t)|^2\mathbb{I} - (x-h(t)) \otimes (x-h(t)) \big] \ dx.
\end{array}
\end{equation*}
(by $\rho_{\calS} >  0 $ we denote the density of the rigid body.)
The normal vector exiting from the fluid domain is denoted by $ n $. Finally $ \F_{0} = \Omega \setminus \calS_{0}$, $ q_{\F,0} $, $ \rho_{\F, 0} $, $ \ell_{0} $ and $ \omega_{0} $ are the initial data.

The goal of this paper is to show that as the size of the object $ \calS $ tends to zero, the associated weak solutions converge up to subsequence in a weak sense to a couple $ (\rho,  u ) $ that satisfies the compressible Navier Stokes equations in all $ \Omega $, which read as
\begin{align}
\partial_t \rho + \DIV( \rho u) = & \, 0 \quad && \text{ for } x \in \Omega, \nonumber \\
\partial_t(\rho u) + \DIV( \rho u \otimes u ) - \DIV\bbS(\nabla u) + \nabla p(\rho)  = & \, 0  \quad && \text{ for } x \in \Omega, \label{equ:CNS}  \\
u = & \, 0 \quad && \text{ for } x \in \partial \Omega. \nonumber
\end{align}

The paper is structured as follow. In the next section we recall an appropriate definitions of weak solutions associated respectivelly with the systems \eqref{equ:CNS:RB} and \eqref{equ:CNS}. In the next one we state the main result and we discuss the hypothesis. Then we introduce a Bogovski\u{\i} operator that follows the rigid body, we show the pressure estimates which is the central part of the article and we conclude by briefly explaining how to pass to the limit in a weak formulation.

\subsection{Definition of weak solutions}

We start by recalling the definition of variational solution for the system \eqref{equ:CNS:RB} from \cite{F:body:comp}, where an existence result was proven.  Actually we decide to slightly change this definition by incorporating the compatibility condition between the rigid motion and the velocity field directly in the spaces as in \cite{GH:exi}, \cite{MB1}, { \cite{KrNePi_2}.} 
 
Let $ \Omega $ { be } an open, bounded, connected subset of $\R^{3}$ with regular boundary. The unknowns of the problem are three namely $ \calS(t) \subset \Omega $ the position of the solid at time $ t $, $ \rho $ the density which describes the density of the fluid in $ \F(t)$, of the rigid body in $ \calS(t) $ and it is extended by zero in $ \R^{3} \setminus\bar{\Omega} $ and $ u $ the velocity field that has to be compatible with the solid velocity in $ \calS(t) $.  

Regarding the pressure law we consider the isentropic pressure-density $ p(\rho) =  \rho^{\gamma} $ and $ \gamma > 3/2 $. Moreover we denote by $ P(\rho) =  \rho^{\gamma}/ (\gamma -1) $. For the  viscosity coefficients we assume $ \mu > 0 $ and $ 3\lambda + 2\mu \geq  0$. Finally we denote the space of the rigid vector field by $ \calR = \{ w: \R^{3} \to \R^{3} $ such that  $ w(x) = l + \omega \times x $ for some $ l$, $ \omega \in \R^{3} \}$. In the following for a measurable set $ A $ we say that $ w: A \to \R^{3} $ is a rigid vector field on $ A $ if there exists $ \tilde{w} \in \calR $ such that $w = \tilde{w}|_{A} $ almost everywhere and with an abuse of notation we write $ w\in \calR$.     
  
The initial data are the position $ \calS_0  \subset \Omega $ of the solid where we assume $ \calS_{0} $ to be open, connected, simply connected, path-connected, measurable, with non empty interior and with Lipschitz boundary. The initial density $ \rho_0 \geq 0  $, which is strictly positive on $ \calS_{0} $. Finally we prescribed the initial momentum $ q_{0} $ such that $  q $ is a.e. identically zero on $ \{ x \in \omega $ such that $ \rho_{0} = 0  \} $ and the restriction of $ q_{0}/\rho_{0}$ on $ \calS_{0}$ is a rigid velocity field on $ \calS_{0}$.     
  
\begin{Definition}
\label{DEF:CNS:RB}
Let $ (\calS_{0}, \rho_{0}, q_{0}) $ an initial data such that $ P(\rho_{0}) \in L^{1}(\Omega) $ and $ |q_0|^2/\rho_0 \in L^{1}(\Omega) $. Then a triple $ (\calS, \rho, u) $ is a weak solution of \eqref{equ:CNS:RB} with initial datum $ (\calS_{0}, \rho_0, q_0) $ for some $ T > 0$ if 
\begin{itemize}

% \item $ \calS: [0,T] \longrightarrow \calA(\Omega) $ where $ \calA(\Omega) $ are the measurable subset of $ \Omega $.

{ \item $\calS(t) \subset \Omega$ is a bounded domain of $\mathbb{R}^3$ for all $t\in [0,T)$ such that 
\begin{equation*}%\label{NO1}
\chi_{\calS}(t,x) = \mathds{1}_{\calS(t)}(x) \in L^{\infty}((0,T) \times \Omega). 
\end{equation*}}

\item $ \rho \in L^{\infty}(0,T; L^{1}(\Omega)) $ such that $ \rho \geq 0 $ and $ P(\rho) \in L^{\infty}(0,T; L^{1}(\Omega)) $.

\item $ u \in \V = \{ u \in L^{2}(0,T;W^{1,2}_{0}(\Omega) ) $ such that $  v|_{\calS(t)} \in \calR    \}$.

\item $ (\calS, \rho, u) $ satisfy the transport equation $ \partial_{t}\rho + \DIV(\rho u) =  0 $ in both a distributional sense in $ [0,T)\times \R^{3} $ and in a renormalized sense where we extend $ \rho $ and $ u $ by zero in the exterior of $ [0,T] \times \Omega $.
{ \item The transport of $\calS$ by the rigid vector field $u_{\calS}$ holds:

\begin{equation*}
\int_0^T\int_{\calS(t)} \partial_t \phi + u_{\calS}\cdot \nabla \phi + \int_{\calS_0} \phi(0,.) = 0,
\end{equation*}
for any $ \phi \in C^{\infty}_c([0,T)\times \mathbb{R}^3)$.}
%\begin{equation*}%\label{NO4}
%\frac{\partial {\chi}_{\calS}}{\partial t} + \operatorname{div}(u_{\calS}\chi_{\calS}) =0 \, \mbox{ in }{\Omega},\quad \chi_{\calS}(t,x)=\mathds{1}_{\calS}(t)(x).
%\end{equation*}}
\item The momentum equation is satisfied in the weak sense
\begin{align*}
\int_{\Omega} q_0 \varphi + \int_{0}^{T} \int_{\Omega} (\rho u)\cdot \partial_t \varphi + [\rho u \otimes u]: D \varphi + p \DIV \varphi = \int_{0}^{T} \int_{\Omega} \bbS u : D \varphi
\end{align*}
for any $ \varphi \in \W $ with
\begin{align*}
\W = \Big\{ \varphi \in C^{\infty}_{c}([0,T)\times \Omega)  \text{ such that }  & \text{for some }\calS_{op} \text{ open neighbourhood of } \bigcup_{t\in [0,T]} \{ t\} \times \overline{\calS(t)} \\ & \text{ it holds } \varphi|_{\calS_{op}} \in C^{\infty}([0,T]; \calR)    \Big\}
\end{align*}

\item For a.e. $ \tau \in [0,T] $ the following energy equality holds
\begin{equation}
\label{ene:est:CNS:RB}
\int_{\Omega} \frac{1}{2}  \rho |u|^2(\tau,.) + P(\rho(\tau,.)) \, dx + \int_{0}^{\tau}\int_{\Omega} \mu | Du|^2+ \lambda | \DIV u |^2 \, dx dt \leq  \int_{\Omega} \frac{1}{2} \frac{|q_0|^2}{\rho_0} +  P(\rho_0) dx.
\end{equation}

\end{itemize}

\end{Definition}

{ \begin{Remark}
Let us mention in work \cite{F:body:comp}
the motion of the rigid body is described through a family of isometries of $\R^3$ by
$$
\eta [t,s]:  \R^3 \to \R^3,\ \ {\overline \calS}(t)= \eta[t,s]({\overline \calS}(s))\ \ \mbox{for}\ 0\leq s\leq t\leq T,
$$
or equivalently
$$
\eta[t,s]=\eta[t,0]\left(\eta[s,0]\right)^{-1},
$$
where the mapping $\eta[t,0]$ satisfies
\[
\eta[t,0](x)\equiv \eta[t](x) = X(t) + \bbO(t) x,\ \ \bbO(t)\in SO(3).
\]

We say that the velocity $u$ is compatible with the family
$\{\calS,\eta\},$  if the
function $t\mapsto \eta[t](x)$ is absolutely continuous
on $[0,T]$ for any $x \in \R^3$ and if
\begin{equation*}
\left( \frac{\partial}{\partial t}\ \eta[t] \right)
\left(\left(\eta[t]\right)^{-1}(x)\right)= u(t,x)\
 \ \ \mbox{for}\ x\in {\overline {\calS}}(t)\ \mbox{and a.e.}\ t\in (0,T).
%\label{etai}
 \end{equation*}
In other words if
$$
u(t,x)=u_{\calS}(t,x)\equiv \ell(t)+{\bbQ}(t)(x-X(t))\ \ \mbox{for}\ x\in {\overline {\calS}}(t)\ \mbox{and a.e.}\ t\in (0,T),$$

 where
\begin{equation*}%\label{eq:QO}
    \ell(t)= \frac{d}{d t} X(t),\ \ \bbQ(t) = \left(\frac{d}{d t}\bbO(t)\right)\left(\bbO(t)\right)^{-1}
\end{equation*}
for a.e. $t \in (0,T)$. Note that the matrix $\bbQ(t)$ is skew-symmetric, so the term $\bbQ(x-X)$ can be written as $\omega \times (x-X)$ for a uniquely determined vector $\omega$.

In the definition of weak solution it is required that the velocity $u$ is compatible with $\{\overline{\calS},\eta\}$ and the functions $\eta[t]: \R^3 \mapsto \R^3$ are isometries
\end{Remark}

%REMARK the extention to a rigid vector field is not unique in the case the solid has zero measure... just pick one and it turns out to be the right one  + no assume Du = 0 only becaouse we do not ...

Let now recall the definiton of variation solution of the system \eqref{equ:CNS} from \cite{NovS}. As before we assume that the initial density $ \rho_0 \geq 0  $ and the initial momentum $ q_{0} $ is such that $  q $ is a.e. identically zero on $ \{ x \in \omega $ such that $ \rho_{0} = 0  \} $.

\begin{Definition}
\label{DEF:CNS}

Let $ (\rho_0, q_0)$ be an initial data such that $ P(\rho_{0}) \in L^{1}(\Omega) $ and $ |q_0|^2/\rho_0 \in L^{1}(\Omega) $. Then a couple $ ( \rho, u) $ is a weak solution of \eqref{equ:CNS} with initial datum $ (\rho_0, q_0) $ for some $ T > 0$ if

\begin{itemize}

\item $ \rho \in L^{\infty}(0,T; L^{1}(\Omega)) $ such that $ \rho \geq 0 $ and $ P(\rho) \in L^{\infty}(0,T; L^{1}(\Omega)) $.

\item $ u \in L^{2}(0,T;W^{1,2}_{0}(\Omega)) $.

\item $ ( \rho, u) $ satisfy the transport equation $ \partial_{t}\rho + \DIV(\rho u) =  0 $ in both a distributional sense in $ [0,T)\times \R^{3} $ and in a renormalized sense where we extend $ \rho $ and $ u $ by zero in the exterior of $ [0,T] \times \Omega $.

\item the momentum equation is satisfied in the weak sense
\begin{align*}
\int_{\Omega} q_0 \varphi + \int_{0}^{T} \int_{\Omega} (\rho u)\cdot \partial_t \varphi + [\rho u \otimes u]: D \varphi + p \DIV \varphi = \int_{0}^{T} \int_{\Omega} \bbS u : D \varphi,
\end{align*}
for any $ \varphi \in C^{\infty}([0,T)\times \Omega)$.

\item for a.e. $ \tau \in [0,T] $ the following energy equality holds
\begin{equation*}
\int_{\Omega} \frac{1}{2}  \rho |u|^2(\tau,.) + P(\rho(\tau,.)) \, dx + \int_{0}^{\tau}\int_{\Omega} \mu | Du|^2+ \lambda | \DIV u |^2 \, dx dt \leq  \int_{\Omega} \frac{1}{2} \frac{|q_0|^2}{\rho_0} +  P(\rho_0) dx.
\end{equation*}

\end{itemize}

\end{Definition} 

In the next section we present the main result of the paper.

\section{Main result and discussion}

The result of this paper read as follow. Let $ \calS_{0} \subset \Omega $ { be } the position of the initial solid, let $ \bar{x} $ { be }such that $ \calS_{0} \subset B_{1}(\bar{x}) $ and let $ \calS_{0,\eps} =  \{ x $ such that $ \bar{x} + (x-\bar{x})/\eps \in \calS_{0} \} $  { be } a sequence of position of initial solid such that $\calS_{0,\eps} \subset \Omega $ and $ \tilde{\rho}_{\eps} > 0 $ their associated density. 
In the following we consider the case where the mass and its angular momentum satisfy the assumptions
\begin{equation}\label{size}
m_{\eps} \eps^{-\frac{3 \gamma-4}{\gamma}} \longrightarrow + \infty \quad \text{ and } \quad \left( \inf_{\xi \in S^2} \xi \cdot \calJ_{\calS_{0,\eps}} \xi \right) \eps ^{-2 - \frac{3 \gamma-4}{\gamma}} \longrightarrow + \infty \quad \text{ as } \eps \longrightarrow 0.
\end{equation}
In particular the mass of the object can converge to zero.

\begin{Theorem}
\label{Theo:main}

Let $ \gamma \geq 6 $. Let $ (\calS_{0,\eps}, \rho_{0,\eps}, q_{0, \eps}) $ { be } a sequence of initial data such that they satisfy  
 
\begin{itemize}
 
\item $ \rho_{0,\eps}|_{\calS_{0,\eps}} = \tilde{\rho}_{\eps} $ and $ \rho_{0,\eps}|_{\F_{0,\eps}} \to \rho_0 $ in $ L^{\gamma}(\Omega) $, where we extend by zero in $ \calS_{0,\eps} $.  

\item $ q_{0,\eps}^2/\rho_{0,\eps} \to q_{0}^2/\rho_{0} $ in $ L^{1}(\Omega) $.
 
\end{itemize}

Let $ (\calS_{\eps}, \rho_{\eps}, u_{\eps}) $  { be } solutions of \eqref{equ:CNS:RB} in the sense of Definition \ref{DEF:CNS:RB} associated with the initial data $ (\calS_{0,\eps}, \rho_{0,\eps}, q_{0, \eps})$. Then up to subsequence there exists $ (\rho, u )$ such that
\begin{equation*}
\rho_{\F,\eps} \longrightarrow \rho \text{ in } C_{w}(0,T; L^{\gamma}(\Omega))  \quad \text{ and } \quad u_{\eps} \cv u \text{ in } L^{2}(0,T; W^{1,2}_{0}(\Omega)). 
\end{equation*} 
Moreover the couple $ (\rho, u ) $ satisfy \eqref{equ:CNS} with initial data $ (\rho_{0}, q_{0}) $ in the sense of Definition \ref{DEF:CNS}.

\end{Theorem}

First of all recall that existence of {local in time weak solution were proved in \cite{DE} and the global existence of weak solution in \cite{F:body:comp}}. The difficulty of this result is to perform improved estimates of the pressure due to the loss of an uniform estimates of the velocity of the rigid body. In particular the dynamics of the rigid body in the limit remains unknown. 

Regarding the hypotesis, to show this result we use a cut-off around the solid and to pass to the limit in the pressure term we need that $ \rho_{\F } $ is in $ L^{q} $ for $ q \geq 3/2 \gamma $ and for this we need the restriction $ \gamma \geq 6 $. Regarding the mass, for incompressible fluid  see \cite{Je:1} and \cite{Je:2} the energy estimates were enought to pass to the limit. Here we need to show the improved regularity of the pressure and in contrast with the case of the fluid alone, where the less integrable term was the convective one $ \rho_{\eps} u_{\eps} \otimes u_{\eps} $, here the worst term is the one coming from the time derivative and can be bounded by a constant times 
\begin{equation*}
\eps^{\frac{5\gamma-\theta-6}{2(\gamma+\theta)}} (|\ell_{\eps}| + \eps|\omega_{\eps}| ),
\end{equation*}
see Proposition \ref{PROP:3}. This shows that to have a control of the $ L^{\gamma +\theta} $ for bigger $ \theta $, we need an better control on the solid velicity which is associated throw the energy estimates to a better control on the  mass $ m_{\eps} $ and the inertia matrix $ \calJ_{\eps} $. For these reasons in our approch appears the theshold 
\begin{equation*}
\frac{3\gamma-4}{\gamma},
\end{equation*} 
which is the one to get $ \rho_{\eps} $ uniformly bounded in $ L^{3\gamma/2}$.

%In fact from one side we need that the pressure is uniformly bounded in $ L^p $ for $ p > 3\gamma/2 $ and on the other we 

\begin{Remark}
It is possible to consider some external forces of the type $ \rho_{\eps} f_{\eps} + g_{\eps }$ in the system \eqref{equ:CNS:RB}. More precisely we can add $ \rho_{\F} f + g $  in the right hand side of the momentum equation (the second equation of \eqref{equ:CNS:RB}), $ \int_{ \calS } \rho_{\calS} f + g  $  and  $ \int_{ \calS }  (x-h(t))\times(\rho_{\calS} f + g)  $ on the right hand side of the Newton's laws associated with respectively  the evolution of the center of mass and the angular rotation (the fifth and the sixth equation of \eqref{equ:CNS:RB}). 

Then Theorem \ref{Theo:main}, Proposition \ref{PROP:1} and \ref{PROP:3} holds also in the presence of external forces if we assume 
\begin{gather*}
f_{\eps} \longrightarrow f \text{ in } L^{2}\left(0,T;L^{\frac{6\gamma}{5\gamma-6}}(\Omega)\right)  \quad \text{ and } \quad  g_{\eps} \longrightarrow g \text{ in } L^{2}\left(0,T;L^{\frac{6}{5}}(\Omega)\right).
\end{gather*}
Note that the gravitation force can be considered, in fact it is enough to choose $ f_{\eps} = \textbf{g} $ the gravitational acceleration.

\end{Remark}

Let us conclude this section with a Corollary. In the case the rigid body is not allowed to move the study of its influence as its size tends to zero is called vanishing obstacle problem. With the use of appropriate cut-off presented in Section \ref{sec:5}, we are able to show that the presence of the small hole does not affect the dynamics of the fluid also for $ \gamma = 6 $. (In the previous works they restict to the case $ \gamma >  6 $). Let us state rigorously the corollary.

\begin{Corollary}
Let $ \gamma \geq 6 $. Let $ \F_{\eps} = \calS_{0,\eps} $ and let $(\rho_{0,\eps}, q_{0,\eps}) $ be a sequence of initial data such that they satisfy

\begin{itemize}
 
\item $ \rho_{0,\eps} \to \rho_0 $ in $ L^{\gamma}(\Omega) $, where we extend by zero in $ \calS_{0,\eps} $.  

\item $ q_{0,\eps}^2/\rho_{0,\eps} \to q_{0}^2/\rho_{0} $ in $ L^{1}(\Omega) $.
 
\end{itemize}

Let $ (\rho_{\eps}, u_{\eps}) $  be solutions of \eqref{equ:CNS} in the domain $ \F_{\eps} $ in the sense of Definition \ref{DEF:CNS} associated with the initial data $ (\rho_{0,\eps}, q_{0, \eps})$. Then up to subsequence there exists $ (\rho, u )$ such that
\begin{equation*}
\rho_{\eps} \longrightarrow \rho \text{ in } C_{w}(0,T; L^{\gamma}(\Omega))  \quad \text{ and } \quad u_{\eps} \cv u \text{ in } L^{2}(0,T; W^{1,2}_{0}(\Omega)). 
\end{equation*} 
Moreover the couple $ (\rho, u ) $ satisfy \eqref{equ:CNS} with initial data $ (\rho_{0}, q_{0}) $ in the sense of Definition \ref{DEF:CNS}.

\end{Corollary}

Note that in a similar way with the use of some appropriate cut-off from Section \ref{sec:5}, it is possible to extend the homogenisation result from \cite{Lu:Schw} to the case $ \gamma = 6 $.

In the remaining we will show Theorem \ref{Theo:main}. We will start by recalling properties of the so called Bogovski\u{\i}, we present the improved estimates for the pressure and finally we show how to pass to the limit in the weak formulation.

\section{Proof of the main theorem}

In this section we show Theorem \ref{Theo:main}. We start by presenting a Bogovski\u{\i} type operator for time dependent domain. We prove a priori estimates for the velocity field and the pressure. Finally we show how to pass to the limit in the weak momentum equation with the help of an appropriate cut-off.

\subsection{The Bogovski\u{\i} operator for time dependent domain}

A key point to show the improved pressure estimates is to be able to invert the divergence operator. For domains with tiny holes this operator was widely studied for example in \cite{Fei:Lu} and \cite{Lu:Schw} where they extended a construction introduced by Allaire in \cite{All}.

The Bogovski\u{\i} operarator associated with the initial solid position $ \calS_{\eps} $ is defined as the composition of three operators an extension operator, a Bogovski\u{\i} operator associated with a subdomain of $ \Omega $ and a restriction operator. In particular we define the extension operator 
\begin{equation*}
\calE_{\eps}: L^{r}(\R^{3} \setminus \calS_{0,\eps}) \longrightarrow L^{r}(\R^3) \quad \text{ where } \calE_{\eps}(f) = f \text{ for } x \in \R^{3} \setminus \calS_{0,\eps} \text{ and } \calE_{\eps}(f) = 0 \text{ elsewhere}.
\end{equation*}
The Bogovski\u{\i} operator $ \calB_{\Omega_1}$ associted with the domain $ \Omega_1 $ as Theorem III.3.2 of \cite{Gal}. And finally the { restriction} operator 
\begin{equation*}
\calR_{\eps}[u] = \eta_{\eps} u + B_{\eps}[\DIV((1-\eta_{\eps}) u) - \ll \DIV (u) \gg_{\calS_{0,\eps}})]  
\end{equation*}
where $ \eta_{\eps}(x-h_{\eps}(0)) = \eta((x-h_{\eps,0})/\eps) $ with $ 1 - \eta \in C^{\infty}_c(B_2(0)) $ such that $ 0 \leq \eta \leq 1 $ and $ 1-\eta = 1 $ in $ B_{1}(0) $, $ B_{\eps} $ a Bogovski\u{\i} operator, from Theorem III.3.2 of \cite{Gal}, associated with the domain $ B_{2\eps}(h_{\eps,0}) \setminus \calS_{\eps,0} $ and 
\begin{equation*}
\ll \DIV u \gg_{ \calS_{0,\eps}} = \frac{1}{|B_{\eps}(0) \setminus \calS_{0,\eps}|} \int_{\calS_{0,\eps}} \DIV u.
\end{equation*}
This definition is analogous to the one presented in \cite{Lu:Schw}. In particular the Bogovski\u{\i} operator reads as  
\begin{equation*}
\calB_{\eps} = \calR_{\eps} \circ \calB \circ \calE_{\eps}: L^{r}(\Omega_{1} \setminus \calS_{0,\eps}) \to W^{1,r}_{0}(\Omega \setminus \calS_{0,\eps})).  
\end{equation*} 
The position of the solid evolves in time and its position can be recover by the position of the center of mass $ h_{\eps} $ and the rotation matrix $ Q_{\eps} $. We define a time dependent Bogovski\u{\i} by consider { extention and restriction operators} that follow the rigid body, in particular 
\begin{gather*}
\calE^{t}_{\eps}[f(t,x)](x) = \calE^0_{\eps}[f(t,h(t_{\eps}+ Q(t)y)](Q^{T}(t)(x-h'(t))), \\ \calR^{t}_{\eps}[F(t,x)](x) = 
\calR^0_{\eps}[F(t,h(t_{\eps}+ Q(t)y)](Q^{T}(t)(x-h'(t))) \\ \text{ and } \\
\calB^t_{\eps} = \calR^t_{\eps} \circ \calB^t \circ \calE^t_{\eps}.
\end{gather*}  

We recall some estimates independent of the small parameter $ \eps $ related to the Bogovski\u{\i} { operator} proved in \cite{Lu:Schw}. 

\begin{Proposition}
\label{PROP:2}
Let $ \Omega_1 \subset \Omega $ with Lipschitz boundary and let $ 1 < q \leq 3 $. Then the linear operator $ \calB^t_{\eps}: L^{q}(\Omega_1 \setminus \calS_{0, \eps}) \to W^{1,q}_{0}(\Omega; \R^{3}) $ such that for any $ f \in L^{q} $ such that $ \int_{\Omega_1 \setminus \calS_{0, \eps}} f = 0 $, it holds
\begin{equation*}
\DIV \calB^t_{\eps}(f) = \calE^t(f) \text{ in } \Omega, \quad \| \calB^t_{\eps}(f) \|_{W^{1,q}(\Omega)} \leq C \|f\|_{L^{q}(\Omega_1 \setminus \calS_{0, \eps})},
\end{equation*}
for some $ C $ independent of $ \eps $. Moreover, 
\begin{equation*}
\| \calB^t_{\eps}(f) \|_{L^{\infty}(\Omega)} \leq C \|f\|_{L^{3}(\Omega_1 \setminus \calS_{0, \eps})}.
\end{equation*}
%

%For $ r > 3/2 $ (I DO NOT UNDERSTAND WHY THEY NEED IT), the linear operator $ \calB^t_{\eps} $ can be extended as a linear operator from $ \{ \DIV g $ such that $ g \in L^{r}(\Omega_1 \setminus \calS_{0, \eps}; \R^3) $ with $ \langle g, 1 \rangle = 0 \}$ to $ L^r(\Omega;\R^3) $ satisfying 
%
%\begin{equation*}
%\| \calB^t_{\eps} (\DIV  g)\|_{L^r(\Omega)} \leq C \|g\|_{L^{r}(\Omega_1 \setminus \calS_{0, \eps})}
%\end{equation*} 
%
%for some constant $ C $ independent on $ \eps $ and $ t $.
\end{Proposition}

Note that the above proposition is Proposition 2.2 of \cite{Lu:Schw} with $ \alpha = 1 $ and in the exponent of the $ \eps $ do not appear the $ -3 $ because it is coming from the fact that the consider a problem with $ \eps^{-3} $ holes.

We are now ready to show the pressure estimates. 

\subsection{Energy estimate and improved pressure estimate}

In this section we prove all the a priori estimates that we need to show Theorem \ref{Theo:main}. 

Let $ (\calS_{\eps}, \rho_{\eps}, u_{\eps}) $ { be } weak solutions of \eqref{equ:CNS:RB} that satisfy the hypothesis of Theorem \ref{Theo:main}. From the energy inequality \eqref{ene:est:CNS:RB}, we deduce that 
\begin{align*}
\int_{\F_{\eps(t)}} \frac{1}{2}  \rho_{\F, \eps} |u_{\F}|^2(t,.) +  & \,  \frac{\rho_{\eps, \F}^{\gamma}(t,.)}{\gamma-1} \, dx +  \frac{1}{2}m_{\eps} |\ell_{\eps}(t)|^2 + \frac{1}{2}\omega_{\eps}(t)\cdot\calJ_{\eps}(t) \omega_{\eps}(t)\\ & \, + \int_{0}^{t}\int_{\F_{\eps}(\tau} \mu | Du_{\F, \eps}|^2+ \lambda | \DIV u_{\F,\eps} |^2 \, dx d\tau \leq  \int_{\Omega} \frac{1}{2} \frac{|q_{\eps, 0}|^2}{\rho_{\eps,0}} +  P(\rho_{\eps, 0}) dx.
\end{align*} 
We deduce that
\begin{gather}
\| \rho_{\F, \eps} u_{\F, \eps} \|_{L^{\infty}(0,T; L^{2}(\F_{\eps}(t)))} \leq C, \nonumber  \\
\|\rho_{\eps,\F}\|_{L^{\infty}(0,T;L^{\gamma}(\F_{\eps}(t)))} \leq C, \nonumber \\
 \eps^{\frac{3 \gamma-4}{2\gamma}} \| \ell_{\eps} \|_{L^{\infty}(0,T)} \longrightarrow  0, \label{en:est} \\
\eps^{1+\frac{3 \gamma-4}{2\gamma}} \| \omega_{\eps} \|_{L^{\infty}(0,T)} \longrightarrow 0,  \nonumber \\
\|u_{\F,\eps}\|_{L^2(0,T;W^{1,2}(\F_{\eps}(t))} \leq C, \nonumber
\end{gather}
where $ C $ is a constant independent of $ \eps $. 

We will use the above estimates to show some better integrability of the pressure. For simplicity from now  we denote by $ \rho_{\eps} $ the extention by zero of the fluid velocity $ \rho_{\F, \eps} $ in the interior of the solid $ \calS_{\eps} $. 
 
\begin{Proposition}[Pressure estimates]
\label{PROP:1}
Let $ \Omega_1 $ { be } a set compactly contained in $ \Omega $ with Lipschitz boundary and let
 $ (\calS_{\eps}, \rho_{\eps}, u_{\eps}) $  { be } solutions of \eqref{equ:CNS:RB} in the sense of Definition \ref{DEF:CNS:RB} associated with the initial data $ (\calS_{0,\eps}, \rho_{\eps}, q_{0, \eps})$. Then 
\begin{equation*}
\int_{0}^{t} \int_{\Omega_1 \setminus \calS_{\eps}(t)} \rho^{3\gamma/2}_{\eps}  \leq C 
\end{equation*}  
for $ t < T $ and $ C $ is independent of $ \eps $ and depends on the $ L^2 $ norm of $ q_0/\sqrt{\rho_0} $ and the $ L^{\gamma} $ norm of $ \rho_0 $. 

\end{Proposition}

\begin{Remark}
The key point is to get the higher integrability of the pressure. In the stationary compressible fluid without structure we can refer to \cite{Fei:Lu}, where they considered the case $\gamma \geq 3$ which guarantees them the $L^2$ integrability of the pressure. Case $3/2<\gamma<3$ was done in the work of \cite{Lu} and more general case in \cite{DFL}. The instationary case is much more delicate problem and it was developed in \cite{Lu:Schw}. For them to get the higher integrability they required $\gamma >6$. In our case we consider particular case with one rigid body which is shrinking and we ``relax'' our assumption on $\gamma \geq 6$. In comparison with incompressible case \cite {Je:1},  we need restriction  on the mass and the angular momentum, see condition \ref{size}.
\end{Remark}

\begin{Remark} Although the above result is enough to prove Theorem \ref{Theo:main}, we also show a uniform bound on the pressures $ \rho_{\eps} $ in $ L^{\gamma + \theta} $ for $ \gamma/2 < \theta \leq 2\gamma/3-1 $, assuming some more restictive hypothesis on the mass $ m_{\eps} $ and the inetia matrix $ \calJ_{\eps} $. 
\end{Remark}

\begin{Proposition} 
\label{PROP:3}
Let $ \gamma > 6 $, let $ \theta \in (\gamma/2,  2\gamma/3-1]$ and suppose that 
\begin{equation}
\label{extra:hyp}
m_{\eps} \geq c_1 \eps^{\frac{5\gamma-\theta-6}{\gamma+\theta}} \quad \text{ and } \quad   \xi \cdot \calJ_{\calS_{0,\eps}} \xi \geq c_2  \eps^{2+ \frac{5\gamma-\theta-6}{\gamma+\theta}} |\xi|^2.
\end{equation}
Let $ \Omega_1 $ a set compactly contained in $ \Omega $ with Lipschitz boundary and let $ (\calS_{\eps}, \rho_{\eps}, u_{\eps}) $ solutions of \eqref{equ:CNS:RB} in the sense of Definition \ref{DEF:CNS:RB} associated with the initial data $ (\calS_{0,\eps}, \rho_{\eps}, q_{0, \eps})$. Then 
\begin{equation*}
\int_{0}^{t} \int_{\Omega_1 \setminus \calS_{\eps}(t)} \rho^{\gamma +\theta}_{\eps}  \leq C 
\end{equation*}  
for $ t < T $ and $ C $ is independent of $ \eps $ and depends on the $ L^2 $ norm of $ q_0/\sqrt{\rho_0} $ and the $ L^{\gamma} $ norm of $ \rho_0 $.

\end{Proposition}

Let us start with the proof of Proposition \ref{PROP:1}.

\subsubsection{Proof of Proposition \ref{PROP:1}}

We follow the classical idea to prove the improved regularity for the pressure, in other words, we test the momentum equation with the Bogovski\u{\i} operator $ \calB_{\eps}^t$ {applied} to $ \rho_{\eps}^{\theta} $. The fact that the Bogovski\u{\i} operator is time dependent { creates} new difficulties and the estimates of this terms are the main novelty.

\begin{proof}[Proof of Proposition \ref{PROP:1}]

Let $ \Omega_2 $ { be } an open set with Lipschitz boundary such that $ \Omega_1 \subset \Omega_{2} \subset \Omega $ and the inclusion are compact. Then we test the momentum equation of \eqref{equ:CNS:RB} with 
\begin{equation*}
\phi \psi \calB_{\eps}^t[ \psi \rho_{\eps}^{\theta}  - \langle   \psi \rho^{\theta} \rangle ] = \phi \psi \calR_{\eps}^t \circ \calB_{\Omega_2} \circ \calE^t [\psi \rho_{\eps}^{\theta} - \langle   \psi \rho^{\theta} \rangle ]
\end{equation*}  
where $ \phi\in C^{\infty}_{c}[0,T)$, $ \psi \in C^{\infty}_{c}(\Omega_2) $ and $ \psi \equiv 1 $ in $ \Omega_1$,  $ \theta = \gamma/2 $ and $ \langle . \rangle $ denotes the average.

We note that actually the test function is not enough smooth in the time variable but  the following estimates can be made rigorous by {convoluting} the test function with a smoothing kernel following the trajectory of the rigid motion and proceed as in Section 7.9.5 of \cite{NovS}.

If we test the momentum equation with $ \phi \psi \calB_{\eps}^t[\psi \rho^{\theta}_{\eps} - \langle   \psi \rho^{\theta}_{\eps} \rangle ] $  we deduce 
\begin{equation}
\label{eq:eq:3}
\int_{0}^{T}\int_{\F_{\eps}(t)} \phi  \psi^2 \rho^{\gamma + \theta}_{\eps} = \sum_{i = 1}^8 J_i,
\end{equation}
where
\begin{gather*}
J_1 = - \int_{0}^{T} \int_{\F_{\eps}(t)}   \rho_{\eps}  u_{\eps}   \phi \psi \partial_{t} \calB_{\eps}^{t}[\psi \rho_{\eps}^{\theta} - \langle \psi \rho_{\eps}^{\theta} \rangle ], \quad 
J_2 = - \int_{0}^{T} \int_{\F_{\eps}(t)}  \rho_{\eps}  u_{\eps}  \psi \phi' \calB_{\eps}^t[\psi \rho_{\eps}^{\theta} - \langle \psi \rho_{\eps}^{\theta} \rangle ], \\
J_3 = \int_{0}^{T} \int_{\F_{\eps}(t)} \phi  \rho_{\eps}   u_{\eps}  \otimes u_{\eps}  : \nabla \left( \psi \calB_{\eps}^t[ \psi \rho_{\eps}^{\theta} - \langle  \psi \rho_{\eps}^{\theta} \rangle ]\right), \quad 
J_{4} = 2\mu \int_{0}^{T} \int_{\F_{\eps}(t) } \phi  D \left( \psi u_{\eps}\right)  : D   \left( \calB_{\eps} ^t [\psi \rho_{\eps}^{\theta} - \langle \psi \rho_{\eps}^{\theta} \rangle ]\right), \quad \\
J_{5}  = \lambda  \int_{0}^{T} \int_{\F_{\eps}(t)}\phi \DIV (u_{\eps} ) \DIV \left(\psi \calB_{\eps}^t[\psi \rho_{\eps}^{\theta} - \langle \psi \rho_{\eps}^{\theta} \rangle ]\right), \quad
J_6 = - \int_{0}^{T}\int_{\F_{\eps}(t)} \phi  \rho^{\gamma}_{\eps} \nabla \psi \calB_{\eps}^t[ \psi \rho_{\eps}^{\theta}  - \langle   \psi \rho^{\theta} \rangle ], \\
J_{7} = \int_{0}^{T} \int_{\F_{\eps}(t)} \rho^{\gamma}_{\eps}  \phi \psi \langle \psi \rho_{\eps}^{\theta} \rangle \quad
\text{ and } \quad  J_{8} = \int_{\F_{\eps}(t) } q_0 \phi(0)\psi(0,.)\calB^0_{\eps} [\psi(0,.) \rho_{\eps}^{\theta}(0,.)- \langle \psi(0,.) \rho_{\eps}^{\theta}(0,.) \rangle]
\end{gather*}
We estimate the right hand side of \eqref{eq:eq:3}, by considering all the terms $ J_i $ separately. We { start} from the easiest one
\begin{align*}
| J_8| \leq \, &\left\| \frac{q_{\eps,0}}{\sqrt{\rho_{\eps,0}}} \right\|_{L^{2}(\F_{\eps}(0))} \| \sqrt{\rho_{\eps,0}} \phi(0)\psi(0,.)\calB_{\eps}^0[\psi(0,.) \rho_{\eps}^{\theta}(0,.)- \langle \psi(0,.) \rho_{\eps}^{\theta}(0,.) \rangle] \|_{L^{2}(\F_{\eps}(0))} \\
\leq \, & \left\| \frac{q_{\eps,0}}{\sqrt{\rho_{\eps,0}}} \right\|_{L^{2}(\F_{\eps}(0))} \| \sqrt{\rho_{\eps,0}}\|_{L^{2\gamma}(\F_{\eps}(0))}
\| \phi(0)\psi(0,.) \calB_{\eps}^0[\psi(0,.) \rho_{\eps}^{\theta}(0,.)- \langle \psi(0,.) \rho_{\eps}^{\theta}(0,.) \rangle]  \|_{L^{2\gamma/(\gamma-1)(\F_{\eps}(0))}} \\
\leq \, & \left\| \frac{q_{\eps,0}}{\sqrt{\rho_{\eps,0}}} \right\|_{L^{2}(\F_{\eps}(0))} \| \sqrt{\rho_{\eps,0}}\|_{L^{2\gamma}(\F_{\eps}(0))} \|\psi(0,.) \rho_{\eps}^{\theta}(0,.) \|_{L^{6\gamma/(5\gamma-3)(\F_{\eps}(0))}},
\end{align*}
which is bounded because $ 6\gamma/(5\gamma-3) < 2 $ for $ \gamma > 3/2 $ and $ \rho_{\eps} \in C_{w}(0,T; L^{\gamma}(\Omega_{\delta}))$.

\begin{align*}
|J_7| \leq & \, C \| \rho_{\eps} \|_{L^{\infty}(0,T;L^{\gamma}(\F_{\eps}(t)))} \int_{0}^{T} \int_{\F_{\eps}(t)} \rho^{\theta}_{\eps}  \\
\leq & \, C \| \rho_{\eps}\|_{L^{\infty}(0,T;L^{\gamma}(\F_{\eps}(t)))}^{\theta},
\end{align*}
where we use $ \theta \leq \gamma $.

\begin{align*}
|J_{6}|  = & \, \left| \int_{0}^{T}\int_{\F_{\eps}(t)} \phi  \rho^{\gamma}_{\eps} \nabla \psi \calB_{\eps}^t[ \psi \rho_{\eps}^{\theta}  - \langle   \psi \rho^{\theta} \rangle ] \right| \\
\leq & \, \| \rho_{\eps} \|_{L^{\infty}(0,T;L^{\gamma}(\F_{\eps}(t))}\| \calB_{\eps}^t[ \psi \rho_{\eps}^{\theta}  - \langle   \psi \rho^{\theta} \rangle ] \|_{L^1(0,T;L^{\infty}(\F_{\eps}(t)))} \\
\leq & \, C \| \psi \rho_{\eps}^{\theta}  - \langle   \psi \rho^{\theta} \rangle \|_{L^3(0,T;L^{3}(\F_{\eps}(t)))} \\
\leq & \, C \| \psi^{1/{\theta}}\rho \|_{L^{3\gamma/2}(0,T;L^{3\gamma/2}(\F_{\eps}(t)))}^{\theta}.
\end{align*} 
where we use $ 2\theta \leq \gamma $.

\begin{align*}
|J_5 | =  & \,\left|\lambda \int_{0}^{T} \int_{\F_{\eps}(t)} \phi \psi \DIV (u_{\eps}) \left(\psi \rho_{\eps}^{\theta} - \langle \psi \rho_{\eps}^{\theta} \rangle \right) + \phi \DIV (u_{\eps}) \nabla \psi \cdot \calB_{\eps}^t[\psi \rho_{\eps}^{\theta} - \langle \psi \rho_{\eps}^{\theta} \rangle]\right| \\
\leq & \, C \|u_{\eps}\|_{L^2(0,T;W^{1,2}_{0}(\F_{\eps}(t)))} \left( \| \rho_{\eps} \|^{\theta}_{L^{\infty}(0,T;L^{\gamma}(\F_{\eps}(t)))} + \|\langle  (\psi \rho_{\eps}^{\theta} \rangle\|_{L^{2}(0,T)} \right) \\
\leq & \, C \|u_{\eps}\|_{L^2(0,T;W^{1,2}_{0}(\F_{\eps}(t)))} \| \rho_{\eps} \|^{\theta}_{L^{\infty}(0,T;L^{\gamma}(\F_{\eps}(t)))},  
\end{align*}
where we use $ 2\theta \leq \gamma $.

\begin{align*}
| J_4 | \leq & \, C \|u_{\eps}\|_{L^2(0,T;W^{1,2}_{0}(\F_{\eps}(t)))} \left\| \phi \nabla (\psi \calB_{\eps}^t [\psi \rho_{\eps}^{\theta} - \langle \psi \rho_{\eps}^{\theta} \rangle ]) \right\|_{L^{2}(0,T;L^{2}(\F_{\eps}(t)))}  \\
\leq & \,  C \|u_{\eps}\|_{L^2(0,T;W^{1,2}_{0}(\F_{\eps}(t)))}\left\| \psi \rho_{\eps}^{\theta} - \langle \psi \rho_{\eps}^{\theta} \rangle \right\|_{L^{2}(0,T;L^{2}(\F_{\eps}(t)))}  \\
\leq & \, C \|u_{\eps}\|_{L^2(0,T;W^{1,2}_{0}(\F_{\eps}(t)))} \| \rho_{\eps} \|^{\theta}_{L^{\infty}(0,T;L^{\gamma}(\F_{\eps}(t)))},
\end{align*}
where we use $ 2\theta \leq \gamma $.

\begin{align*}
| J_3 |\leq  C \|u_{\eps}\|_{L^2(0,T;W^{1,2}_{0}(\F_{\eps}(t)))}^2 \left\| \rho_{\eps} \phi \nabla(\psi \calB_{\eps}^t [\psi \rho_{\eps}^{\theta} -  \langle \psi \rho_{\eps}^{\theta} \rangle ]) \right\|_{L^{\infty}(0,T;L^{3/2}(\F_{\eps}(t)))}.
\end{align*}
we estimate the last term as
\begin{align*}
\left\|\rho_{\eps} \nabla (\psi \calB_{\eps}^t[\psi \rho_{\eps}^{\theta} -  \langle \psi \rho_{\eps}^{\theta} \rangle ]) \right\|_{L^{3/2}(\F_{\eps}(t))} \leq \,  & \|\rho_{\eps}\|_{L^{\gamma}(\F_{\eps}(t))} \left\|\nabla (\psi \calB_{\eps}^t[\psi \rho_{\eps}^{\theta} -  \langle  \psi \rho_{\eps}^{\theta} \rangle )] \right\|_{L^{3\gamma/(2\gamma-3)}(\F_{\eps(t)})} \\ 
\leq \,  &  \|\rho_{\eps}\|_{L^{\gamma}(\F_{\eps}(t))} \|\psi \rho_{\eps}^{\theta} -  \langle \psi \rho_{\eps}^{\theta} \rangle\|_{L^{3\gamma/(2\gamma-3)}(\F_{\eps}(t))} \\
\leq \, & C\|\rho_{\eps}\|_{L^{3\gamma\theta /(2\gamma-3)}(\F_{\eps}(t))}^{\theta},
\end{align*}
we conclude that 
\begin{equation*}
| J_3 |\leq  C \|u_{\eps}\|_{L^2(0,T;W^{1,2}_{0}(\F_{\eps}(t)))}^2 \|\rho_{\eps} \|_{L^{\infty}(0,T;L^{\gamma}(\F_{\eps}(t)))} \|\rho_{\eps}\|_{L^{\infty}(0,T;L^{3\gamma\theta /(2\gamma-3)}(\F_{\eps}(t)))}^{\theta},
\end{equation*}
note that $|J_3| $ is bounded assuming $ \theta \leq 2 \gamma/ 3-1$.

\begin{align*}
|J_2| \leq \, & C\|\rho_{\eps}u_{\eps}\|_{L^2(0,T;L^3(\F_{\eps}(t)))}\|\phi'\|_{L^2(0,T)}\| \calB_{\eps}^t[\psi \rho_{\eps}^{\theta} - \langle \psi \rho_{\eps}^{\theta} \rangle ] \|_{L^{\infty}(0,T;L^{3/2}(\F_{\eps}(t)))} \\
\leq \, & C \|\rho_{\eps} u_{\eps}\|_{L^2(0,T;L^3(\F_{\eps}(t)))}\|\phi'\|_{L^2(0,T)}\| \psi \rho_{\eps}^{\theta} - \langle \psi \rho_{\eps}^{\theta} \rangle \|_{L^{\infty}(0,T;L^{1}(\F_{\eps}(t)))}  \\
\leq \, & C\|\rho_{\eps} u_{\eps}\|_{L^2(0,T;L^3(\F_{\eps}(t)))}\|\phi'\|_{L^2(0,T)}\|\rho^{\theta}_{\eps} \|_{L^{\infty}(0,T;L^{1}(\F_{\eps}(t)))}, 
\end{align*}
in partcular it is bounded for $ \theta \leq \gamma $.

We move to the most difficult term which is the one involving the time derivative Bogovski\u{\i}. Recall that $  \rho_{\eps} $ satisfies a transport equation in a renomalizad sense for which we can use $ b(.) = |.|^{\theta} $ with $ \theta = \gamma/2 $ as test function due to Lemma 6.9 of \cite{NovS} apply with $ \beta = \gamma $, $ \lambda_{1} = \gamma/2 -1 $. Moreover $ \psi \rho_{\eps}^{\theta} $ satisfy
\begin{equation}
\label{tra:equ:cut}
\partial_t (\psi \rho_{\eps}^{\theta}) + \DIV (\psi \rho_{\eps}^{\theta} u_{\eps} ) + (\theta -1 ) \psi \rho_{\eps}^{\theta} \DIV u_{\eps} -  \rho_{\eps} \nabla \psi \cdot u_{\eps} = 0. 
\end{equation}
Before showing the estimates of $ J_1 $, we compute the time derivative of the Bogovski\u{\i} operator $ \calB_{\eps}^t $. Let now $ f $  {\ be  a sufficiently } smooth function then 
\begin{align*}
\partial_t \calB_{\eps}^t[f] = & \, \partial_t \left( \calR_{\eps}^t \circ \calB_{\Omega_2} \circ \calE^t_{\eps} [f]\right) \\
= & \, \partial_t \left( \calR_{\eps}^0[B_{\Omega_2}[\calE^t[f]](h(t)+Q(t) y) ](Q^T(t)(x-h(t)))  \right) \\
= & \, \calR_{\eps}^t\left[\partial_{t}B_{\Omega_2}[\calE^t_{\eps}[f]] + (u_{S,\eps}\cdot \nabla) B_{\Omega_2}[\calE^t_{\eps}[f]] \right] - u_{S,\eps}\cdot \nabla \calR_{\eps}^t [ \calB_{\Omega_2} [ \calE^t_{\eps} [f]]] \\ 
= & \, \calR_{\eps}^t\left[\partial_{t}B_{\Omega_2}[\calE^t_{\eps}[f]]\right] + \calR_{\eps}^t \left[(u_{S,\eps}\cdot \nabla) B_{\Omega_2}[\calE^t_{\eps}[f]] \right] - u_{S,\eps}\cdot \nabla \calR_{\eps}^t [ \calB_{\Omega_2} [ \calE^t_{\eps} [f]]].  
\end{align*}   
Moreover the function $ \calE_{\eps}^{t}(\psi \rho^{\theta}_{\eps}) = \psi \rho_{\F,\eps}^{\theta} $ which satisfies \eqref{tra:equ:cut} in $ [0,T] \times \R^{3} $.

We rewrite the term $ \partial_t \calB_{\eps}^t[\psi \rho_{\F,\eps}^{\theta} - \langle \psi \rho_{\F,\eps}^{\theta} \rangle ] $ taking advantage of some cancellation.
For $ f $ with zero average, we denote by $ g = \calE^t_{\eps}[f] $ and exploting the definition of $ \calR_{\eps}^t $
\begin{align*}
 \calR_{\eps}^t \left[(u_{S,\eps}\cdot \nabla) \calB_{\Omega_2}[g] \right] & - u_{S,\eps}\cdot \nabla \calR_{\eps}^t [ \calB_{\Omega_2} [g]]  \\
= \,  & \eta_{\eps} (u_{S,\eps}\cdot \nabla) \calB_{\Omega_2}[g] + B_{\eps}[ \DIV((1-\eta_{\eps})(u_{S,\eps}\cdot \nabla)  \calB_{\Omega_2}[g]) - \ll \DIV((u_{S,\eps}\cdot \nabla) \calB_{\Omega_2}[g]) \gg_{\calS_{\eps}(t)}] \\
& - (u_{S,\eps}\cdot \nabla)\eta_{\eps} \cdot \calB_{\Omega_2}[g] - \eta_{\eps} (u_{S,\eps}\cdot \nabla) \calB_{\Omega_2}[g] \\
& - (u_{S,\eps}\cdot \nabla) B_{\eps}[\DIV((1-\eta_{\eps})\calB_{\Omega_2}[g])-\ll \DIV(\calB_{\Omega_2}[g]) \gg_{\calS_{\eps}(t)}] \\
= \, &  B_{\eps}[ \DIV((1-\eta_{\eps})(u_{S,\eps}\cdot \nabla)  \calB_{\Omega_2}[g]) - \ll \DIV((u_{S,\eps}\cdot \nabla) \calB_{\Omega_2}[g]) \gg_{\calS_{\eps}(t)}] - (u_{S,\eps}\cdot \nabla)\eta_{\eps} \cdot \calB_{\Omega_2}[g] \\ & 
- (u_{S,\eps}\cdot \nabla) B_{\eps}[\DIV((1-\eta_{\eps})\calB_{\Omega_2}[g])] \\
\end{align*}
We rewrite the first term of the last expression. To do that we introduce the notation 
\begin{equation*}
(U_{\eps})_{i,j} = \partial_i(u_{S,\eps})_j =  \partial_i(\omega_{\eps} \times x)_{j}.
\end{equation*}
We have
\begin{align}
\DIV((1-\eta_{\eps})(u_{S,\eps}\cdot \nabla)  \calB_{\Omega_2}[g]) = & \, -\nabla \eta_{\eps} \cdot (u_{S,\eps}\cdot \nabla)  \calB_{\Omega_2}[g]) + (1-\eta_{\eps})U_\eps : \nabla  \calB_{\Omega_2}[g]] \nonumber \\
& \, (1-\eta_{\eps})u_{S,\eps}\cdot \nabla g \nonumber \\
=  & \, -\nabla \eta_{\eps} \cdot (u_{S,\eps}\cdot \nabla)  \calB_{\Omega_2}[g]) + (1-\eta_{\eps})U_\eps : \nabla  \calB_{\Omega_2}[g] \nonumber \\
& \,   g \nabla \eta_{\eps} \cdot u_{S,\eps}  +  u_{S,\eps}\cdot \nabla ((1-\eta_{\eps})g).  \label{canc:equ:fin}
\end{align}
Similarly
\begin{align*}
\DIV((1-\eta_{\eps})\calB_{\Omega_2}[g]) = - \nabla \eta_{\eps}\cdot \calB_{\Omega_2}[g] + (1-\eta_{\eps})g.
\end{align*}
Let now consider 
\begin{align}
\calR_{\eps}^t\left[\calB_{\Omega_2}[\partial_{t}(\psi \rho^{\theta}_{\eps}- \langle \psi \rho^{\theta}_{\eps} \rangle)]\right] 
= & \, \eta_{\eps} \calB_{\Omega_2}\left[ \partial_t(\psi \rho^{\theta}_{\eps}- \langle \psi \rho^{\theta}_{\eps} \rangle) \right] \nonumber \\ & \, + B_{\eps}[\DIV((1-\eta_{\eps})\calB_{\Omega_2}[\partial_{t}(\psi \rho^{\theta}_{\eps}- \langle \psi \rho^{\theta}_{\eps} \rangle)]) - \ll\DIV(\calB_{\Omega_2}[\partial_{t}(\psi \rho^{\theta}_{\eps}- \langle \psi \rho^{\theta}_{\eps} \rangle)]) \gg ] \nonumber \\
= & \, \eta_{\eps} \calB_{\Omega_2}\left[ \partial_t(\psi \rho^{\theta}_{\eps}- \langle \psi \rho^{\theta}_{\eps} \rangle) \right] + B_{\eps}\Big[ - \nabla \eta_{\eps} \cdot \calB_{\Omega_2}[\partial_{t}(\psi \rho^{\theta}_{\eps}- \langle \psi \rho^{\theta}_{\eps} \rangle)])  \nonumber \\
& \, \quad \quad   \quad  \quad  \quad  \quad  \quad \quad  \quad  \quad \quad + (1-\eta_{\eps})\partial_{t}(\psi \rho^{\theta}_{\eps}- \langle \psi \rho^{\theta}_{\eps} \rangle) - \ll \partial_{t}(\psi \rho^{\theta}_{\eps} - \langle \psi \rho^{\theta}_{\eps} \rangle)) \gg \Big]. \label{ave:term:also} 
\end{align}
Using the fact that $ \psi \rho^{\theta}_{\eps} $ satisfies \eqref{tra:equ:cut}, the first term of \eqref{ave:term:also} reads
\begin{align*}
(1-\eta_{\eps})\partial_{t}(\psi \rho^{\theta}_{\eps}- \langle \psi \rho^{\theta}_{\eps} \rangle) = & \, - (1-\eta_{\eps})(\DIV (\psi \rho_{\eps}^{\theta} u_{\eps}) \\ & \, - (1-\eta_{\eps})((\theta -1 ) \psi \rho_{\eps}^{\theta} \DIV u_{\eps} - \langle (\theta -1 ) \psi \rho_{\eps}^{\theta} \DIV u_{\eps} \rangle) \\ & \, + (1-\eta_{\eps})(  \rho_{\eps} \nabla \psi \cdot u_{\eps}  - \langle \rho_{\eps} \nabla \psi \cdot u_{\eps} \rangle) \\
= & \, -\DIV ((1-\eta_{\eps})\psi \rho_{\eps}^{\theta} u_{\eps} ) - \nabla \eta_{\eps} \psi \rho_{\eps}^{\theta} u_{\eps}  \\ & \, -(1-\eta_{\eps})((\theta -1 ) \psi \rho_{\eps}^{\theta} \DIV u_{\eps} - \langle (\theta -1 ) \psi \rho_{\eps}^{\theta} \DIV u_{\eps} \rangle) \\ & \, + (1-\eta_{\eps})(  \rho_{\eps} \nabla \psi \cdot u_{\eps}  - \langle \rho_{\eps} \nabla \psi \cdot u_{\eps} \rangle). 
\end{align*}
Finally we notice that the first term of the right hand side of the above equality together with the last term of \eqref{canc:equ:fin} read
\begin{equation*}
- \DIV ((1-\eta_{\eps})\psi \rho_{\eps}^{\theta} u_{\eps} ) +  u_{S,\eps}\cdot \nabla ((1-\eta_{\eps})\psi \rho_{\eps}^{\theta} u_{\eps})] =  -\DIV ((1-\eta_{\eps})\psi \rho_{\eps}^{\theta} (u_{\eps}-u_{S,\eps} )).
\end{equation*}
Analogously we have 
\begin{gather*}
- \ll \partial_{t}(\psi \rho^{\theta}_{\eps} - \langle \psi \rho^{\theta}_{\eps} \rangle)) \gg  - \ll \DIV((u_{S,\eps}\cdot \nabla) \calB_{\Omega_2}[\psi \rho^{\theta}_{\eps} - \langle \psi \rho^{\theta}_{\eps}\rangle]) \gg_{\calS_{\eps}(t)} 
 = \ll \DIV(\psi \rho^{\theta}_{\eps} (u_{\eps}- u_{S,\eps})  \gg  \\
+ \ll (\theta -1 ) \psi \rho_{\eps}^{\theta} \DIV u_{\eps} - \langle  (\theta -1 ) \psi \rho_{\eps}^{\theta} \DIV u_{\eps} \rangle \gg 
- \ll \rho_{\eps} \nabla \psi \cdot u_{\eps} - \langle \rho_{\eps} \nabla \psi \cdot u_{\eps} \rangle  \gg  = 0,
\end{gather*}
because any term is zero. We deduce that 
\begin{align*}
B_{\eps}[ & - \nabla \eta_{\eps} \cdot \calB_{\Omega_2}[\partial_{t}(\psi \rho^{\theta}_{\eps}- \langle \psi \rho^{\theta}_{\eps} \rangle)])  + \DIV((1-\eta_{\eps})\calB_{\Omega_2}[\partial_{t}(\psi \rho^{\theta}_{\eps}- \langle \psi \rho^{\theta}_{\eps} \rangle)]) - \ll\DIV(\calB_{\Omega_2}[\partial_{t}(\psi \rho^{\theta}_{\eps}- \langle \psi \rho^{\theta}_{\eps} \rangle)]) \gg ] \\ & \, + B_{\eps}[ \DIV((1-\eta_{\eps})(u_{S,\eps}\cdot \nabla)  \calB_{\Omega_2}[g]) - \ll \DIV((u_{S,\eps}\cdot \nabla) \calB_{\Omega_2}[g]) \gg_{\calS_{\eps}(t)}] \\ 
= & \,  B_{\eps}\Big[  -\DIV ((1-\eta_{\eps})\psi \rho_{\eps}^{\theta} (u_{\eps}-u_{S,\eps} ))  
- \nabla \eta_{\eps} \psi \rho_{\eps}^{\theta} u_{\eps}\\ &
 -(1-\eta_{\eps})((\theta -1 ) \psi \rho_{\eps}^{\theta} \DIV u_{\eps} - \langle (\theta -1 ) \psi \rho_{\eps}^{\theta} \DIV u_{\eps} \rangle) 
+ (1-\eta_{\eps})(  \rho_{\eps} \nabla \psi \cdot u_{\eps}  - \langle \rho_{\eps} \nabla \psi \cdot u_{\eps} \rangle)  \\ &
-\nabla \eta_{\eps} \cdot (u_{S,\eps}\cdot \nabla)  \calB_{\Omega_2}[\psi \rho^{\theta}_{\eps}- \langle \psi \rho^{\theta}_{\eps} \rangle]) 
+ (1-\eta_{\eps})U_\eps : \nabla  \calB_{\Omega_2}[\psi \rho^{\theta}_{\eps}- \langle \psi \rho^{\theta}_{\eps} \rangle]   + \psi \rho^{\theta}_{\eps} \nabla \eta_{\eps} \cdot u_{S,\eps}
\Big].
\end{align*}
In a compact way, we can write
\begin{equation*}
\partial_t \calB_{\eps}[\psi \rho^{\theta}_{\eps}- \langle \psi \rho^{\theta}_{\eps} \rangle ] = \sum_{i = 1}^{11} n_i,
\end{equation*}
where
\begin{gather*}
n_1 =  \eta_{\eps} \calB_{\Omega_2}\left[ \partial_t(\psi \rho^{\theta}_{\eps}- \langle \psi \rho^{\theta}_{\eps} \rangle) ]\right], \quad 
n_2 = - B_{\eps}[\nabla \eta_{\eps} \cdot \calB_{\Omega_2}[\partial_{t}(\psi \rho^{\theta}_{\eps}- \langle \psi \rho^{\theta}_{\eps} \rangle)]) - \lll \nabla \eta_{\eps} \cdot \calB_{\Omega_2}[\partial_{t}(\psi \rho^{\theta}_{\eps}- \langle \psi \rho^{\theta}_{\eps} \rangle)]) \ggg ], \\
n_3 =  - B_{\eps}[\DIV ((1-\eta_{\eps})\psi \rho_{\eps}^{\theta} (u_{\eps}-u_{S,\eps} ))], \quad
n_4 = - B_{\eps}[\nabla \eta_{\eps} \cdot u_{\eps}\psi \rho_{\eps}^{\theta} - \lll  \nabla \eta_{\eps} \cdot u_{\eps}\psi \rho_{\eps}^{\theta} \ggg ], \\
n_5 = - B_{\eps}[(1-\eta_{\eps})((\theta -1 ) \psi \rho_{\eps}^{\theta} \DIV u_{\eps} - \langle (\theta -1 ) \psi \rho_{\eps}^{\theta} \DIV u_{\eps} \rangle) - \lll (1-\eta_{\eps})((\theta -1 ) \psi \rho_{\eps}^{\theta} \DIV u_{\eps} - \langle (\theta -1 ) \psi \rho_{\eps}^{\theta} \DIV u_{\eps} \rangle) \ggg], \\
n_6 =  B_{\eps}[(1-\eta_{\eps})(  \rho_{\eps} \nabla \psi \cdot u_{\eps}  - \langle \rho_{\eps} \nabla \psi \cdot u_{\eps} \rangle) - \lll (1-\eta_{\eps})(  \rho_{\eps} \nabla \psi \cdot u_{\eps}  - \langle \rho_{\eps} \nabla \psi \cdot u_{\eps} \rangle) \ggg], \\
n_7 = - (u_{S,\eps}\cdot \nabla)B_{\eps}[(1-\eta_{\eps})(\psi \rho^{\theta}_{\eps}-\langle \psi \rho^{\theta}_{\eps} \rangle )- \lll (1-\eta_{\eps})(\psi \rho^{\theta}_{\eps}-\langle \psi \rho^{\theta}_{\eps} \rangle ) \ggg],  \\
n_8 = (u_{S, \eps} \cdot \nabla)B_{\eps}[\nabla \eta_{\eps}\cdot \calB_{\Omega_2}[\psi \rho^{\theta}_{\eps} - \langle \psi \rho^{\theta}_{\eps} \rangle ]- \lll \nabla \eta_{\eps}\cdot \calB_{\Omega_2}[\psi \rho^{\theta}_{\eps} - \langle \psi \rho^{\theta}_{\eps} \rangle ] \ggg], \\
n_9 = B_{\eps}[\psi \rho^{\theta}_{\eps} \nabla \eta_{\eps} \cdot u_{S,\eps} - \lll \psi \rho^{\theta}_{\eps} \nabla \eta_{\eps} \cdot u_{S,\eps}  \ggg ],  \\
n_{10} =  B_{\eps}[(1-\eta_{\eps})U_\eps : \nabla  \calB_{\Omega_2}[\psi \rho^{\theta}_{\eps}- \langle \psi \rho^{\theta}_{\eps} \rangle ]- \lll (1-\eta_{\eps})U_\eps : \nabla  \calB_{\Omega_2}[\psi \rho^{\theta}_{\eps}- \langle \psi \rho^{\theta}_{\eps} \rangle ] \ggg] \\
n_{11} = B_{\eps}[-\nabla \eta_{\eps} \cdot (u_{S,\eps}\cdot \nabla)  \calB_{\Omega_2}[\psi \rho^{\theta}_{\eps} - \langle \psi \rho^{\theta}_{\eps} \rangle ])-\lll -\nabla \eta_{\eps} \cdot (u_{S,\eps}\cdot \nabla)  \calB_{\Omega_2}[\psi \rho^{\theta}_{\eps} - \langle \psi \rho^{\theta}_{\eps} \rangle ]) \ggg], \\ \text{ and } 
n_{12} = - (u_{S,\eps}\cdot \nabla)\eta_{\eps} \cdot \calB_{\Omega_2}[\psi \rho^{\theta}_{\eps}- \langle \psi \rho^{\theta}_{\eps} \rangle - \lll \psi \rho^{\theta}_{\eps}- \langle \psi \rho^{\theta}_{\eps} \rangle \ggg ],
\end{gather*}
and $ \lll . \ggg $ denotes the average on $ B_{\eps}(h_{\eps}(t)) \setminus \calS_{\eps}(t) $.
We have that 
\begin{equation*}
J_1 = - \int_{0}^{T} \int_{\F_{\eps}(t)}  \phi \psi \rho_{\eps}  u_{\eps}   \phi \partial_{t} \calB_{\eps}^{t}[\psi \rho_{\eps}^{\theta} - \langle \psi \rho_{\eps}^{\theta} \rangle ] = - \int_{0}^{T} \int_{\F_{\eps}(t)}  \psi \rho_{\eps}  u_{\eps}   \phi \sum_{i=1}^{12} n_i = - \sum_{i=1}^{12} B_i, 
\end{equation*}
with
\begin{equation*}
B_i = \int_{0}^{T} \int_{\F_{\eps}(t)}  \sqrt{\phi} \psi \rho_{\eps}  u_{\eps}   \sqrt{\phi} n_i.
\end{equation*}
We now estimate the terms $ B_i $ separately. Denote by $\tilde{B}_{\eps}[f] = B_{\eps}[f- \lll f \ggg]$. Recall that 
\begin{equation*}
\|\sqrt{\phi} \psi \rho_{\eps} u_{\eps} \|_{L^{\frac{6\gamma}{3\gamma+4}}\left(0,T;L^{\frac{6\gamma}{\gamma+4}}(\F_{\eps}(t))\right)} \leq \| \sqrt{\phi} \psi  \rho_{\eps}\|_{L^{3\gamma/2}\left(0,T;L^{3\gamma/2}(\F_{\eps}(t))\right)}\|u_{\eps}\|_{L^{2}\left(0,T;L^{6}(\F_{\eps}(t))\right)}
\end{equation*}
and note that
\begin{align*}
|B_{i}| \leq \|\sqrt{\phi} \psi \rho_{\eps} u_{\eps}\|_{L^{\frac{6\gamma}{3\gamma+4}}\left(0,T;L^{\frac{6\gamma}{\gamma+4}}(\F_{\eps}(t))\right)}\|\sqrt{\phi} n_i\|_{L^{\frac{6\gamma}{3\gamma-4}}\left(0,T;L^{\frac{6\gamma}{5\gamma-4}}(\F_{\eps}(t))\right)}.
\end{align*}
It remains to estimates $ \sqrt{\phi} n_i $ for $ i = 1,...12$.
Let us start with $i=12$.
\begin{align*}
\|\sqrt{\phi} n_{12}\|_{L^{6\gamma/(5\gamma-4)}(\F_{\eps}(t))} \leq & \, (|\ell_{\eps}| + \eps |\omega_{\eps}|)\|\nabla \eta_{\eps}\|_{L^{6\gamma/(5\gamma-4)}(\F_{\eps}(t))} \| \sqrt{\phi}\overline{\calB_{\Omega_2}}[\psi \rho^{\theta}_{\eps}- \langle \psi \rho^{\theta}_{\eps} \rangle ] \|_{L^{\infty}(\F_{\eps}(t))} \\
\leq & \, (|\ell_{\eps}| + \eps |\omega_{\eps}|) \eps^{(3\gamma-4)/2\gamma} \| \sqrt{\phi}  \psi \rho^{\theta}\|_{L^{3}(\F_{\eps}(t))},
\end{align*}
{
we apply Brezis- Bourgain theorem, (see also Remark III.3.7, \cite{Gal}),  properties of Bogovskii operator \cite{Gal} and from the behavior of $\eta$ : $\|\eta\|_p \approx \epsilon ^{\frac{3-p}{p}}$.}
We deduce
\begin{align*}
|B_{12}| \leq C \| \phi^{1/{2\theta}} \psi^{1/\theta} \rho_{\eps}\|_{L^{3\gamma/2}(0,T;L^{3\gamma/2}(\F_{\eps}(t)))}^{\theta}.
\end{align*}

$i=11$

We have
\begin{align*}
\|\sqrt{\phi} n_{11} \|_{L^{6\gamma/(5\gamma-4)}(\F_{\eps}(t))} = & \, \|\phi \tilde{B}_{\eps}[-\nabla \eta_{\eps} \cdot (u_{S,\eps}\cdot \nabla)  \calB_{\Omega_2}[\psi \rho^{\theta}_{\eps} - \langle \psi \rho^{\theta}_{\eps} \rangle ])] \|_{L^{6\gamma/(5\gamma-4)}(\F_{\eps}(t))} \\
\leq & \, \eps \|\sqrt{\phi} \tilde{B}_{\eps}[-\nabla \eta_{\eps} \cdot (u_{S,\eps}\cdot \nabla)  \calB_{\Omega_2}[\psi \rho^{\theta}_{\eps} - \langle \psi \rho^{\theta}_{\eps} \rangle ])] \|_{L^{6\gamma/(3\gamma-4)}(\F_{\eps}(t))} \\
\leq & \, \eps \| \sqrt{\phi} \nabla \eta_{\eps} \cdot (u_{S,\eps}\cdot \nabla)  \calB_{\Omega_2}[\psi \rho^{\theta}_{\eps} - \langle \psi \rho^{\theta}_{\eps} \rangle ]) \|_{L^{6\gamma/(5\gamma-4)}(\F_{\eps}(t))} \\
\leq & \, (|\ell_{\eps}|+\eps|\omega_{\eps}|)\eps^{1+ \frac{\gamma  - 4}{2\gamma}}\| \sqrt{\phi} \nabla \calB_{\Omega_2}[\psi \rho^{\theta}_{\eps} - \langle \psi \rho^{\theta}_{\eps} \rangle ]) \|_{L^{3}(\F_{\eps}(t))} \\
\leq & \, (|\ell_{\eps}|+\eps|\omega_{\eps}|)\eps^{(3\gamma  - 4)/2\gamma}\| \sqrt{\phi} \psi \rho^{\theta}_{\eps} \|_{L^{3}(\F_{\eps}(t))}.
\end{align*}
{ where we apply the behavior of the rigid velocity,  the cut off function $\eta$ and Remark 3.19, $(3.3.35)_3$ \cite {NovS}.}
\begin{align*}
|B_{11}| \leq C \| \phi^{1/{2\theta}} \psi^{1/\theta} \rho_{\eps}\|_{L^{3\gamma/2}(0,T;L^{3\gamma/2}(\F_{\eps}(t)))}^{\theta}.
\end{align*}
Similarly we have 
\begin{align*}
\|\sqrt{\phi} n_{10} \|_{L^{6\gamma/(5\gamma-4)}(\F_{\eps}(t))}  = & \, \|\sqrt{\phi} \tilde{B}_{\eps}[(1-\eta_{\eps})U_\eps : \nabla  \calB_{\Omega_2}[\psi \rho^{\theta}_{\eps}- \langle \psi \rho^{\theta}_{\eps} \rangle ]] \|_{L^{6\gamma/(5\gamma-4)}(\F_{\eps}(t))}  \\
\leq & \, \eps \|\sqrt{\phi} \tilde{B}_{\eps}[(1-\eta_{\eps})U_\eps : \nabla  \calB_{\Omega_2}[\psi \rho^{\theta}_{\eps}- \langle \psi \rho^{\theta}_{\eps} \rangle ]] \|_{L^{6\gamma/(3\gamma-4)}(\F_{\eps}(t))} \\
\leq & \, \eps \|\sqrt{\phi} (1-\eta_{\eps}) U_\eps : \nabla   \calB_{\Omega_2}[\psi \rho^{\theta}_{\eps}- \langle \psi \rho^{\theta}_{\eps} \rangle ] \|_{L^{6\gamma/(5\gamma-4)}(\F_{\eps}(t))} \\
\leq & \, \eps|\omega_{\eps}|\eps^{3\frac{3\gamma-4}{6\gamma}}\| \sqrt{\phi}\nabla   \calB_{\Omega_2}[\psi \rho^{\theta}_{\eps}- \langle \psi \rho^{\theta}_{\eps} \rangle ]\|_{L^{3}(\F_{\eps}(t))} \\
\leq & \, \eps|\omega_{\eps}|\eps^{3\gamma-6/2\gamma}\| \sqrt{\phi} \psi \rho^{\theta}_{\eps} \|_{L^{3}(\F_{\eps}(t))}.
\end{align*}
We deduce
\begin{align*}
|B_{10}| \leq C \| \phi^{1/{2\theta}} \psi^{1/\theta}\rho_{\eps}\|_{L^{3\gamma/2}(0,T;L^{3\gamma/2}(\F_{\eps}(t)))}^{\theta}.
\end{align*}
We have 
\begin{align*}
\|\sqrt{\phi} n_{9} \|_{L^{6\gamma/(5\gamma-4)}(\F_{\eps}(t))}  = & \, \|\sqrt{\phi} \tilde{B}_{\eps}[(\psi \rho^{\theta}_{\eps}- \langle \psi \rho^{\theta}_{\eps} \rangle  ) \nabla \eta_{\eps} \cdot u_{S,\eps}  ] \|_{L^{6\gamma/(5\gamma-4)}(\F_{\eps}(t))} \\
\leq & \, \eps \| \sqrt{\phi} (\psi \rho^{\theta}_{\eps}- \langle \psi \rho^{\theta}_{\eps} \rangle  ) \nabla \eta_{\eps} \cdot u_{S,\eps} \|_{L^{6\gamma/(5\gamma-4)}(\F_{\eps}(t))} \\
\leq & \, (|\ell_{\eps}| + \eps |\omega_{\eps}| )\eps^{1+\frac{\gamma-4}{2\gamma}} \|\sqrt{\phi} \psi \rho^{\theta}_{\eps} \|_{L^{3}(\F_{\eps}(t))}.
\end{align*}
We deduce
\begin{align*}
|B_{9}| \leq C \| \phi^{1/{2\theta}} \psi^{1/\theta}\rho_{\eps}\|_{L^{3\gamma/2}(0,T;L^{3\gamma/2}(\F_{\eps}(t)))}^{\theta}.
\end{align*}
We have
\begin{align*}
\|\sqrt{\phi} n_{8} \|_{L^{6\gamma/(5\gamma-4)}(\F_{\eps}(t))} \leq & \, \| \sqrt{\phi}  (u_{S, \eps} \cdot \nabla)\tilde{B}_{\eps}[\nabla \eta_{\eps}\cdot \calB_{\Omega_2}[\psi \rho^{\theta}_{\eps} - \langle \psi \rho^{\theta}_{\eps} \rangle ]]  \|_{L^{6\gamma/(5\gamma-4)}(\F_{\eps}(t))} \\
\leq & \, (|\ell_{\eps}| + \eps |\omega_{\eps}|)\| \sqrt{\phi} \nabla \eta_{\eps}\cdot \calB_{\Omega_2}[\psi \rho^{\theta}_{\eps} - \langle \psi \rho^{\theta}_{\eps} \rangle  \|_{L^{6\gamma/(5\gamma-4)}(\F_{\eps}(t))} \\
\leq & \,  (|\ell_{\eps}| + \eps |\omega_{\eps}|)\| \nabla \eta_{\eps} \|_{L^{6\gamma/(5\gamma-4)}(\F_{\eps}(t))}\| \sqrt{\phi} \calB_{\Omega_2}[\psi \rho^{\theta}_{\eps} - \langle \psi \rho^{\theta}_{\eps} \rangle \|_{L^{\infty(\F_{\eps})}} \\
\leq & \, (|\ell_{\eps}|+\eps|\omega_{\eps}|)\eps^{(3\gamma  - 4)/2\gamma}\| \sqrt{\phi} \psi \rho^{\theta}_{\eps} \|_{L^{3}(\F_{\eps}(t))}.
\end{align*}
We deduce
\begin{align*}
|B_{8}| \leq C \| \phi^{1/{2\theta}} \psi^{1/\theta} \rho_{\eps}\|_{L^{3\gamma/2}(0,T;L^{3\gamma/2}(\F_{\eps}(t)))}^{\theta}.
\end{align*}
We have
\begin{align*}
\|\sqrt{\phi} n_{7} \|_{L^{6\gamma/(5\gamma-4)}(\F_{\eps}(t))} \leq & \, \|\sqrt{\phi} (u_{S,\eps}\cdot \nabla)\tilde{B}_{\eps}[(1-\eta_{\eps})(\psi \rho^{\theta}_{\eps}-\langle \psi \rho^{\theta}_{\eps} \rangle )] \|_{L^{6\gamma/(5\gamma-4)}(\F_{\eps}(t))} \\
\leq & \, (|\ell_{\eps}|+\eps|\omega_{\eps}|)\|\sqrt{\phi}(1-\eta_{\eps})(\psi \rho^{\theta}_{\eps}-\langle \psi \rho^{\theta}_{\eps} \rangle ) \|_{L^{6\gamma/(5\gamma-4)}(\F_{\eps}(t))} \\
\leq & \, (|\ell_{\eps}|+\eps|\omega_{\eps}|)\eps^{3\frac{3\gamma  - 4}{6\gamma}}\| \sqrt{\phi} \psi \rho^{\theta}_{\eps} \|_{L^{3}(\F_{\eps}(t))}.
\end{align*}
We deduce
\begin{align*}
|B_{7}| \leq C \| \phi^{1/{2\theta}}\psi^{1/\theta} \rho_{\eps}\|_{L^{3\gamma/2}(0,T;L^{3\gamma/2}(\F_{\eps}(t)))}^{\theta}.
\end{align*}
Recall that 
\begin{equation*}
%\label{est:8pm}
\|\rho_{\eps} u_{\eps} \|_{L^{2}(0,T;L^{3}(\F_{\eps}(t)))} \leq C,
\end{equation*}
follows from the energy estimate and { $ \gamma \geq 6$.} This allows us to estimate 
\begin{equation*}
|B_{6}| \leq \|\rho_{\eps} u_{\eps} \|_{L^{\infty}(0,T;L^{3}(\F_{\eps}(t)))}\| \phi n_6 \|_{L^{2}(0,T;L^{3/2}(\F_{\eps}(t)))}.
\end{equation*} 
We have
\begin{align*}
\|\phi n_{6} \|_{L^{3/2}(\F_{\eps}(t))} \leq & \, \|\phi \tilde{B}_{\eps}[(1-\eta_{\eps})(  \rho_{\eps} \nabla \psi \cdot u_{\eps}  - \langle \rho_{\eps} \nabla \psi \cdot u_{\eps} \rangle)]\|_{L^{3/2}(\F_{\eps}(t))}  \\
\leq & \, \eps^{2/5} \|\phi \tilde{B}_{\eps}[(1-\eta_{\eps})(  \rho_{\eps} \nabla \psi \cdot u_{\eps}  - \langle \rho_{\eps} \nabla \psi \cdot u_{\eps} \rangle)]\|_{ L^{2}(\F_{\eps}(t))}  \\ 
\leq & \, \eps^{2/5}\|\phi (1-\eta_{\eps})(  \rho_{\eps} \nabla \psi \cdot u_{\eps}  - \langle \rho_{\eps} \nabla \psi \cdot u_{\eps} \rangle) \|_{L^{6/5}(\F_{\eps}(t))}  \\
\leq & \eps^{2/5} \, \| u_{\eps}\|_{L^{6}(\F_{\eps}(t))}\| \phi \rho^{\theta}_{\eps} \|_{L^{\frac{3}{2}}(\F_{\eps}(t))}.
\end{align*}
%
%{\footnote{How you get 6/5?}}
We deduce that for $ \theta \leq 2\gamma/3 $
\begin{align*}
|B_{6}| \leq C \| \rho_{\eps}\|_{L^{\infty}(0,T;L^{\gamma}(\F_{\eps}(t)))}^{\theta}.
\end{align*}
Recall that by energy estimates we have
\begin{equation*}
\|\rho_{\eps} u_{\eps} \|_{L^{2}(0,T;L^{\frac{6\gamma}{6+\gamma}}(\F_{\eps}(t)))} \leq C \quad \text{ and } \quad \|\rho_{\eps} u_{\eps} \|_{L^{\infty}(0,T;L^{2\gamma/(\gamma+1)}(\F_{\eps}(t)))} \leq C.
\end{equation*}
Using { $ \gamma \geq 6 $} and interpolation we have %{\footnote{It is also true in the case $\gamma \geq 6$?}}
\begin{equation*}
\|\rho_{\eps} u_{\eps} \|_{L^{6}(0,T;L^{2}(\F_{\eps}(t)))} \leq \|\rho_{\eps} u_{\eps} \|_{L^{2}(0,T;L^{3}(\F_{\eps}(t)))}^{1/3}
\|\rho_{\eps} u_{\eps} \|_{L^{\infty}(0,T;L^{12/7}(\F_{\eps}(t)))}^{2/3}.
\end{equation*}
%{\footnote{It is correct exponent $3/2$, I think that it should be $2/3$}.} 
\begin{equation*}
|B_{5}| \leq \|\rho_{\eps} u_{\eps} \|_{L^{6}(0,T;L^{2}(\F_{\eps}(t)))}\| \phi n_5 \|_{L^{6/5}(0,T;L^{2}(\F_{\eps}(t)))}.
\end{equation*} 
\begin{align*}
\|\phi n_{5} \|_{L^{2}(\F_{\eps}(t))} \leq & \, \| \phi \tilde{B}_{\eps}[(1-\eta_{\eps})((\theta -1 ) \psi \rho_{\eps}^{\theta} \DIV u_{\eps} - \langle (\theta -1 ) \psi \rho_{\eps}^{\theta} \DIV u_{\eps} \rangle)]
\|_{L^{2}(\F_{\eps}(t))}  \\
\leq & \, \| (1-\eta_{\eps})((\theta -1 ) \psi \rho_{\eps}^{\theta} \DIV u_{\eps} \|_{L^{6/5}(\F_{\eps}(t))}  \\
\leq & \, \| \DIV(u_{\eps})\|_{L^{2}(\F_{\eps}(t))} \| \phi \rho^{\theta}_{\eps} \|_{L^{3}(\F_{\eps}(t))}.
\end{align*}
We deduce
\begin{align*}
|B_{5}| \leq C \| \phi^{1/{\theta}} \rho_{\eps}\|_{L^{3\gamma/2}(0,T;L^{3\gamma/2}(\F_{\eps}(t)))}^{\theta}.
\end{align*}
We have
\begin{equation*}
|B_{4}| \leq \|\rho_{\eps} u_{\eps} \|_{L^{2}(0,T;L^{6\gamma/(6+\gamma)}(\F_{\eps}(t)))}\| \phi n_4 \|_{L^{2}(0,T;L^{6\gamma/(5\gamma-6)}(\F_{\eps}(t)))}.
\end{equation*} 
Moreover
\begin{align*}
\|\phi n_{4} \|_{L^{6\gamma/(5\gamma-6)}(\F_{\eps}(t))} \leq & \, \| \phi 
\tilde{B}_{\eps}[\nabla \eta_{\eps} \cdot u_{\eps}\psi \rho_{\eps}^{\theta} ] \|_{L^{6\gamma/(5\gamma-6)}(\F_{\eps}(t))}  \\
\leq & \, \eps \| \phi 
\tilde{B}_{\eps}[\nabla \eta_{\eps} \cdot u_{\eps}\psi \rho_{\eps}^{\theta} ]\|_{L^{2\gamma/(\gamma-2)}(\F_{\eps}(t))} \\
\leq & \,  \eps \| \phi \nabla \eta_{\eps} \cdot u_{\eps}\psi \rho_{\eps}^{\theta}  \|_{L^{6\gamma/(5\gamma-6)}(\F_{\eps}(t))}  \\
\leq & \, \| u_{\eps} \|_{L^{6}(\F_{\eps}(t))}\|\psi \rho_{\eps}^{\theta} \|_{L^{3\gamma/(2\gamma-3)}(\F_{\eps}(t))}.
\end{align*}
We deduce, for $ \theta < 2\gamma/3-1 $
\begin{align*}
|B_{4}| \leq C \| \rho_{\eps}\|_{L^{\infty}(0,T;L^{\gamma}(\F_{\eps}(t)))}^{\theta}.
\end{align*}
In $ B_3 $ there is the term $ u_{\eps}-u_{\eps,S} $ ans $ u_{\eps} $ and $ u_{\eps,S} $ have different integrability in time. To prove the estimate we introduce a smooth cut-off $ \chi_{\eps}: B_{\eps}(h_{\eps}(t))\setminus\calS_{\eps}(t) \to [0,1]$  such that
\begin{equation} \label{cuttoff}
\begin{array}{l}
% \chi_{\eps}: B_{\eps}(h_{\eps}(t))\setminus\calS_{\eps}(t) \to [0,1] \mbox{ such that } \\ \\ 
\chi_{\eps} = 1  \mbox{ in an open neighborhood of } \partial \calS_{\eps}(t) ,\\ \\
 \chi_{\eps} = 0 \mbox { in an open neighborhood of } \partial B_{\eps}(h_{\eps}(t)) \\ \\ 
|{\mbox{supp}}(\chi_{\eps})| \leq \eps^{(3\gamma-4)/(\gamma-2)}.

\end{array}
\end{equation}
%\end{proof} 
%\end{document}
We have
\begin{align*}
n_3 = & \, - B_{\eps}[\DIV ((1-\eta_{\eps})\psi \rho_{\eps}^{\theta} (u_{\eps}-u_{S,\eps} ))]  \\
= & \, - B_{\eps}[\DIV ((1-\eta_{\eps}-\chi_{\eps})\psi \rho_{\eps}^{\theta} u_{\eps})]+B_{\eps}[\DIV ((1-\eta_{\eps}-\chi_{\eps})\psi \rho_{\eps}^{\theta} u_{S,\eps} )] {+} B_{\eps}[\DIV (\chi_{\eps}\psi \rho_{\eps}^{\theta} (u_{\eps}-u_{S,\eps} ))] \\
= & \, n_3^1+n_3^2+n_3^3.
\end{align*}  
It holds 
\begin{align*}
|B_3| \leq & \, \| \rho_{\eps} u_{\eps}\|_{L^{2}(0,T;L^{6\gamma/(\gamma+6)}(\F_{\eps}(t)))}\left(\|\phi n_{3}^1 + \phi n_{3}^3\|_{L^{2}(0,T;L^{6\gamma/(5\gamma-6)}(\F_{\eps}(t)))}\right) \\
& \, \|\sqrt{\phi} \psi \rho_{\eps} u_{\eps}\|_{L^{\frac{6\gamma}{3\gamma+4}}\left(0,T;L^{\frac{6\gamma}{\gamma+4}}(\F_{\eps}(t))\right)}\|\sqrt{\phi} n_3^2\|_{L^{\frac{6\gamma}{3\gamma-4}}\left(0,T;L^{\frac{6\gamma}{5\gamma-4}}(\F_{\eps}(t))\right)}.
\end{align*}
Moreover
\begin{align*}
\|\phi n_{3}^1 \|_{L^{2}(0,T;L^{6\gamma/(5\gamma-6)}(\F_{\eps(t)}))} \leq  & \, \| \phi \psi  \rho_{\eps}^{\theta}u_{\eps} \|_{L^{2}(0,T;L^{6\gamma/(5\gamma-6)}(\F_{\eps}(t)))} \\
\leq & \, \|u_{\eps}\|_{_{L^{2}(0,T;L^{6}(\F_{\eps}(t)))}}\|\rho_{\eps}\|_{L^{\infty}(0,T;L^{\gamma}(\F_{\eps}(t)))}^{\theta},
\end{align*}
which holds for $ \theta \leq 2\gamma/3-1 $.
\begin{align*}
\|\phi n_{3}^3 \|_{L^{2}(0,T;L^{6\gamma/(5\gamma-6)}(\F_{\eps(t)}))} \leq  & \, \| \phi \psi  \chi_{\eps}\rho_{\eps}^{\theta}(u_{\eps}-u_{s,\eps}) \|_{L^{2}(0,T;L^{6\gamma/(5\gamma-6)}(\F_{\eps(t)}))} \\
\leq & \, \|u_{\eps}\|_{_{L^{2}(0,T;L^{6}(\F_{\eps}(t)))}}\|\rho_{\eps}\|_{L^{\infty}(0,T;L^{\gamma}(\F_{\eps}(t)))} \\ & + (\|\ell_{\eps}\|_{L^{\infty}(0,T)}+\eps\|\omega_{\eps}\|_{L^{\infty}(0,T)}\|\chi_{\eps}\|_{L^{\infty}(0,T;L^{2\gamma/(\gamma-2)}(\F_{\eps}(t)))}\| \phi \psi \rho^{\theta}\|_{L^{2}(0,T;L^{3}(\F_{\eps}(t)))}\\
\leq & \, \|u_{\eps}\|_{_{L^{2}(0,T;L^{6}(\F_{\eps}(t)))}}\|\rho_{\eps}\|_{L^{\infty}(0,T;L^{\gamma}(\F_{\eps}(t)))} \\
 & \, + (\|\ell_{\eps}\|_{L^{\infty}(0,T)}+\eps\|\omega_{\eps}\|_{L^{\infty}(0,T)})\eps^{\frac{3\gamma-4}{\gamma-2}\frac{\gamma-2}{2\gamma}}\|\phi \psi \rho_{\eps}^{\theta}\|_{L^{2}(0,T;L^{3}(\F_{\eps}(t)))}.
\end{align*}
\begin{align*}
\|\sqrt{\phi} n_3^2 \|_{L^{\frac{6\gamma}{5\gamma-4}}(\F_{\eps}(t))} \leq & \, \|\sqrt{\phi}(1-\eta_{\eps}-\chi_{\eps}) \psi \rho_{\eps}^{\theta} u_{S,\eps}  \|_{L^{\frac{6\gamma}{5\gamma-4}}(\F_{\eps}(t))} \\
\leq & \,(|\ell_{\eps}|+\eps|\omega_{\eps}|) \eps^{(3\gamma-4)/2\gamma} \| \sqrt{\phi} \psi \rho_{\eps}^{\theta} \|_{L^{3}(\F_{\eps}(t))}.
\end{align*}
We deduce
\begin{align*}
|B_3| \leq C\left( \| \rho_{\eps}\|_{L^{\infty}(0,T;L^{\gamma}(\F_{\eps}(t)))}^{\theta} +   \| \phi^{1/{2\theta}} \psi^{1/\theta}\rho_{\eps}\|_{L^{3\gamma/2}(0,T;L^{3\gamma/2}(\F_{\eps}(t)))}^{\theta}   \right).
\end{align*}
It is well-known that for $ p < 6 $ 
\begin{align}
\label{gen:est}
\|\phi  \calB_{\Omega_2}[ \partial_t(\psi \rho^{\theta}_{\eps}- \langle \psi \rho^{\theta}_{\eps} \rangle) ] \|_{L^p(\F_{\eps}(t))} \leq  \left( \|u_{\eps} \|_{L^{6}(\F_{\eps}(t))}  + \|\DIV(u_{\eps})\|_{L^{2}(\F_{\eps}(t))}  \right)\| \phi \rho_{\eps}^{\theta}\|_{L^{6p/(6-p)}(\F_{\eps}(t))}.
\end{align}
Then 
\begin{align*}
|B_{2}| \leq \|\rho_{\eps} u_{\eps} \|_{L^{6}(0,T;L^{2}(\F_{\eps}(t)))}\| \phi n_2 \|_{L^{6/5}(0,T;L^{2}(\F_{\eps}(t)))}.
\end{align*}
Moreover
\begin{align*}
\|\phi n_{2} \|_{L^{2}(\F_{\eps}(t))} \leq & \, \| \phi 
 \tilde{B}_{\eps}[\nabla \eta_{\eps} \cdot \calB_{\Omega_2}[\partial_{t}(\psi \rho^{\theta}_{\eps}- \langle \psi \rho^{\theta}_{\eps} \rangle)]) ] \|_{L^{2}(\F_{\eps}(t))}  \\
 \leq & \, \eps \| \phi 
 \tilde{B}_{\eps}[\nabla \eta_{\eps} \cdot \calB_{\Omega_2}[\partial_{t}(\psi \rho^{\theta}_{\eps}- \langle \psi \rho^{\theta}_{\eps} \rangle)]) ]\|_{L^{6}(\F_{\eps}(t))} \\
 \leq & \,  \eps \| \phi \nabla \eta_{\eps} \cdot \calB_{\Omega_2}[\partial_{t}(\psi \rho^{\theta}_{\eps}- \langle \psi \rho^{\theta}_{\eps} \rangle)]) \|_{L^{2}(\F_{\eps}(t))}  \\
\leq & \, \left( \| u_{\eps}\|_{L^{6}(\F_{\eps}(t))} + \|\DIV(u_{\eps})\|_{L^{2}(\F_{\eps}(t))} \right)\| \phi \psi\rho_{\eps}^{\theta} \|_{L^{3}(\F_{\eps}(t))},
\end{align*}
where we use \eqref{gen:est} with $ p = 2 $. We deduce
\begin{align*}
|B_{2}| \leq C \| \phi^{2/\gamma} \psi^{2/\gamma}\rho_{\eps}\|_{L^{3\gamma/2}(0,T;L^{3\gamma/2}(\F_{\eps}(t)))}^{\theta}.
\end{align*}
We have  
\begin{align*}
|B_{1}| \leq & \, \|\rho_{\eps} u_{\eps} \|_{L^{6}(0,T;L^{2}(\F_{\eps}(t)))}\| \phi n_1 \|_{L^{6/5}(0,T;L^{2}(\F_{\eps}(t)))} \\
\leq & \, \|\rho_{\eps} u_{\eps} \|_{L^{6}(0,T;L^{2}(\F_{\eps}(t)))}\| \phi  \eta_{\eps} \calB_{\Omega_2}\left[ \partial_t(\psi \rho^{\theta}_{\eps}- \langle \psi \rho^{\theta}_{\eps} \rangle) ]\right] \|_{L^{6/5}(0,T;L^{2}(\F_{\eps}(t)))} \\
\leq & \, \|\rho_{\eps} u_{\eps} \|_{L^{6}(0,T;L^{2}(\F_{\eps}(t)))}\left( \| u_{\eps}\|_{L^2(0,T;L^{6}(\F_{\eps}(t))} + \|\DIV(u_{\eps})\|_{L^{2}((0,T)\times\F_{\eps}(t))} \right)\| \phi \psi \rho_{\eps}^{\theta} \|_{L^{3}((0,T)\times \F_{\eps}(t))} \\
\leq & \, C \| \phi^{2/\gamma} \psi^{2/\gamma}\rho_{\eps}\|_{L^{3\gamma/2}(0,T;L^{3\gamma/2}(\F_{\eps}(t)))}^{\theta}.
\end{align*}
Putting all the estimates together and recalling that $ \theta = \gamma / 2 $ we deduce that \eqref{eq:eq:3} reads
\begin{equation*}
\int_{0}^{T}\int_{\F_{\eps}(t)} \phi \psi^2 \rho_{\eps}^{3\gamma/2} = \sum_{i = 2}^{7} J_i + \sum_{i = 1}^{12} B_i \leq C + \left(\int_{0}^{T}\int_{\F_{\eps}} \phi^{3/2} \psi^{3} \rho_{\eps}^{3\gamma/2}\right)^{1/3+2/3\gamma} 
\end{equation*}
\end{proof}

\subsubsection{Proof of Proposition \ref{PROP:3}}

From Proposition \ref{PROP:1} we deduce that $ \rho_{\eps} $ is uniformly bounded in $ L^{3\gamma/2}((0,t)\times \Omega_1 ) $. We will use this information and the extra hypothesis \eqref{extra:hyp} to show that  $ \rho_{\eps} $ is uniformly bounded in $ L^{\gamma + \theta}((0,t)\times \Omega_1 ) $ for $ \theta $ in $ (\gamma/2, 2\gamma/3-1] $.

\begin{proof}[Proof of Proposition \ref{PROP:3}.]
As in Proposition \ref{PROP:1}, we test the momentum equation with 
\begin{equation*}
\phi \psi \calB_{\eps}^t[ \psi \rho_{\eps}^{\theta}  - \langle   \psi \rho^{\theta} \rangle ] = \phi \psi \calR_{\eps}^t \circ \calB_{\Omega_2} \circ \calE^t [\psi \rho_{\eps}^{\theta} - \langle   \psi \rho^{\theta} \rangle ],
\end{equation*}
for $ \theta  \in (\gamma/2, 2\gamma/3-1] $.
As in Proposition \ref{PROP:1} we deduce that 
\begin{equation*}
\int_{0}^{T}\int_{\F_{\eps}(t)} \phi  \psi^2 \rho^{\gamma + \theta}_{\eps} = \sum_{i = 1}^{8} J_i,
\end{equation*}
Note that the term $ J_2$, $ J_3$, $ J_7 $, $ J_8 $ can be estimates in the same way as Proposition \ref{PROP:1}. We are left with $ J_1$, $ J_4$, $ J_5 $ and $ J_6 $. We start with $ J_8 $
\begin{align*}
\left| \int_{0}^{T}\int_{\F_{\eps}(t)} \phi  \rho^{\gamma}_{\eps} \nabla \psi \calB_{\eps}^t[ \psi \rho_{\eps}^{\theta}  - \langle   \psi \rho^{\theta} \rangle ] \right| \leq & \, \| \phi\rho_{\eps}^{\gamma} \|_{L^{3/2}((0,T)\times \Omega_2)}\| \calB_{\eps}^t[ \psi \rho_{\eps}^{\theta}  - \langle   \psi \rho^{\theta} \rangle ] \|_{L^3(0,T;L^3(\F_{\eps}(t)))} \\
\leq & \, C \| \psi \rho_{\eps}^{\theta}  - \langle   \psi \rho^{\theta} \rangle \|_{L^3(0,T;L^{3/2}(\F_{\eps}(t)))} \\
\leq & \, C \| \rho \|_{L^{\infty}(0,T;L^{\gamma}(\F_{\eps}(t)))}^{\theta},
\end{align*}  
for $ 3 \theta/2 < \gamma $.

Moreover
\begin{align*}
|J_4 + J_5 | \leq & \, 2 \| \nabla u_{\eps}\|_{L^2(0,T;L^{2}(\F_{\eps}(t)))}\| \nabla(\psi \calB_{\eps}^t[\psi \rho_{\eps}^{\theta} - \langle \psi \rho_{\eps}^{\theta} \rangle ])\|_{L^2(0,T;L^{2}(\F_{\eps}(t)))} \\
\leq & \, C \|\psi \rho_{\eps}^{\theta} - \langle \psi \rho_{\eps}^{\theta} \rangle \|_{L^2(0,T;L^{2}(\F_{\eps}(t)))} \\
\leq & \, C \| \psi^{1/{\theta}}\rho_{\eps}\|_{L^{2\theta}(0,T;L^{2\theta}(\F_{\eps}(t)))}^{\theta} \\
\leq & C \| \psi^{1/{\theta}}\rho_{\eps}\|_{L^{3\gamma/2}(0,T;L^{3\gamma/2}(\F_{\eps}(t)))}^{\theta},
\end{align*}
for $ 2 \theta < 3\gamma/2$.
We are left with $ J_1 $.

We show the estimates for $ J_1 $. To this aim we estimates the $ B_i $ terms. It holds
\begin{equation*}
\|\sqrt{\phi} \psi \rho_{\eps} u_{\eps} \|_{L^{\frac{2(\gamma+\theta)}{2 + \gamma+\theta}}\left(0,T;L^{\frac{6(\gamma+\theta)}{6+\gamma+\theta}}(\F_{\eps}(t))\right)} \leq \| \sqrt{\phi} \psi  \rho_{\eps}\|_{L^{\gamma+\theta}\left(0,T;L^{\gamma + \theta}(\F_{\eps}(t))\right)}\|u_{\eps}\|_{L^{2}\left(0,T;L^{6}(\F_{\eps}(t))\right)},
\end{equation*}
we deduce that
\begin{align*}
|B_{i}| \leq  \|\sqrt{\phi} \psi \rho_{\eps} u_{\eps} \|_{L^{\frac{2(\gamma+\theta)}{2 + \gamma+\theta}}\left(0,T;L^{\frac{6(\gamma+\theta)}{6+\gamma+\theta}}(\F_{\eps}(t))\right)} \|\sqrt{\phi} n_i\|_{L^{\frac{2(\gamma+\theta)}{\gamma+\theta - 2}}\left(0,T;L^{\frac{6(\gamma+\theta)}{5(\gamma+\theta)-6}}(\F_{\eps}(t))\right)}.
\end{align*}
We estimate the terms $\sqrt{\phi}n_i $. Let notice that the inequality used for $ n_{12} $ and $ n_8 $, for $ n_{11} $ and $ n_9 $ and for $ n_{10} $, $ n_7 $ and $ n_{3}^2 $ are the same so we only show the estimate for $ n_{12} $, $ n_{11} $ and $ n_{10} $.
It holds that 
 \begin{align*}
\|\sqrt{\phi} n_{12}\|_{L^{\frac{6(\gamma+\theta)}{5(\gamma+\theta)-6}}(\F_{\eps}(t))} \leq & \, (|\ell_{\eps}| + \eps |\omega_{\eps}|)\|\nabla \eta_{\eps}\|_{L^{\frac{6(\gamma+\theta)}{7\gamma+\theta-6}}(\F_{\eps}(t))} \| \sqrt{\phi}\calB_{\Omega_2}[\psi \rho^{\theta}_{\eps}- \langle \psi \rho^{\theta}_{\eps} \rangle ] \|_{L^{\frac{3(\gamma+\theta)}{2\theta-\gamma}}(\F_{\eps}(t))} \\
\leq & \, (|\ell_{\eps}| + \eps |\omega_{\eps}|) \eps^{\frac{5\gamma-\theta-6}{2(\gamma+\theta)}} \| \sqrt{\phi}  \psi \rho^{\theta}_{\eps}\|_{L^{\frac{\gamma+\theta}{\theta}}(\F_{\eps}(t))}
\end{align*}
In the estimate above we use the identity
\begin{gather*}
\frac{2\theta-\gamma}{3(\gamma+\theta)} = \frac{\theta}{\gamma+\theta}-\frac{1}{3} \\
\left( 3- \frac{6(\gamma+\theta)}{7\gamma+\theta-6} \right)\frac{7\gamma+\theta-6}{6(\gamma+\theta)} = \frac{5\gamma-\theta-6}{2(\gamma+\theta)}
\end{gather*}
We deduce
\begin{align*}
|B_{12}| \leq C  (\|\ell_{\eps}\|_{L^{\infty}(0,T)} + \eps |\omega_{\eps}|_{L^{\infty(0,T)}}) \eps^{\frac{5\gamma-\theta-6}{2(\gamma+\theta)}} \| \phi^{1/{2\theta}} \psi^{1/\theta} \rho_{\eps}\|_{L^{\gamma+\theta}(0,T;L^{\gamma+\theta}(\F_{\eps}(t)))}^{\theta}.
\end{align*}
We have
\begin{align*}
\|\sqrt{\phi} n_{11} \|_{L^{\frac{6(\gamma+\theta)}{5(\gamma+\theta)-6}}(\F_{\eps}(t))}  = & \, \|\sqrt{\phi} \tilde{B}_{\eps}[-\nabla \eta_{\eps} \cdot (u_{S,\eps}\cdot \nabla)  \calB_{\Omega_2}[\psi \rho^{\theta}_{\eps} - \langle \psi \rho^{\theta}_{\eps} \rangle ])] \|_{L^{\frac{6(\gamma+\theta)}{5(\gamma+\theta)-6}}(\F_{\eps}(t))} \\
\leq & \, \eps \|\sqrt{\phi} \tilde{B}_{\eps}[-\nabla \eta_{\eps} \cdot (u_{S,\eps}\cdot \nabla)  \calB_{\Omega_2}[\psi \rho^{\theta}_{\eps} - \langle \psi \rho^{\theta}_{\eps} \rangle ])] \|_{L^{\frac{6(\gamma+\theta)}{3\gamma+3\theta-6}}(\F_{\eps}(t))} \\
\leq & \, \eps \| \sqrt{\phi} \nabla \eta_{\eps} \cdot (u_{S,\eps}\cdot \nabla)  \calB_{\Omega_2}[\psi \rho^{\theta}_{\eps} - \langle \psi \rho^{\theta}_{\eps} \rangle ]) \|_{L^{\frac{6(\gamma+\theta)}{5(\gamma+\theta)-6}}(\F_{\eps}(t))} \\
\leq & \, (|\ell_{\eps}|+\eps|\omega_{\eps}|)\eps^{1+ \frac{3\gamma-3\theta  - 6}{2(\gamma+\theta)}}\| \sqrt{\phi} \nabla \calB_{\Omega_2}[\psi \rho^{\theta}_{\eps} - \langle \psi \rho^{\theta}_{\eps} \rangle ]) \|_{L^{\frac{\gamma+\theta}{\theta}}(\F_{\eps}(t))} \\
\leq & \, (|\ell_{\eps}|+\eps|\omega_{\eps}|)\eps^{\frac{5\gamma-\theta-6}{2(\gamma+\theta)}}\| \sqrt{\phi} \psi \rho^{\theta}_{\eps} \|_{L^{\frac{\gamma+\theta}{\theta}}(\F_{\eps}(t))},
\end{align*}
where we used
\begin{gather*}
\frac{5(\gamma+\theta)-6}{6(\gamma+\theta)} = \frac{\theta}{\gamma+\theta}+\frac{5\gamma-\theta-6}{6(\gamma+\theta)}, \\
\left(3- \frac{6(\gamma+\theta)}{5\gamma-\theta-6}\right)\frac{5\gamma-\theta-6}{6(\gamma+\theta)} =  \frac{3\gamma+\theta  - 6}{2(\gamma+\theta)}.
\end{gather*}
We deduce
\begin{align*}
|B_{11}| \leq C  (\|\ell_{\eps}\|_{L^{\infty}(0,T)} + \eps |\omega_{\eps}|_{L^{\infty(0,T)}}) \eps^{\frac{5\gamma-\theta-6}{2(\gamma+\theta)}}  \| \phi^{1/{2\theta}} \psi^{1/\theta} \rho_{\eps}\|_{L^{\gamma+\theta}(0,T;L^{\gamma+\theta}(\F_{\eps}(t)))}^{\theta}.
\end{align*}
Similarly we have 
\begin{align*}
\|\sqrt{\phi} n_{10} \|_{L^{\frac{6(\gamma+\theta)}{5(\gamma+\theta)-6}}(\F_{\eps}(t))}   = & \, \|\sqrt{\phi} \tilde{B}_{\eps}[(1-\eta_{\eps})U_\eps : \nabla  \calB_{\Omega_2}[\psi \rho^{\theta}_{\eps}- \langle \psi \rho^{\theta}_{\eps} \rangle ]] \|_{L^{\frac{6(\gamma+\theta)}{5(\gamma+\theta)-6}}(\F_{\eps}(t))}  \\
\leq & \, \eps \|\sqrt{\phi} \tilde{B}_{\eps}[(1-\eta_{\eps})U_\eps : \nabla  \calB_{\Omega_2}[\psi \rho^{\theta}_{\eps}- \langle \psi \rho^{\theta}_{\eps} \rangle ]] \|_{L^{\frac{6(\gamma+\theta)}{3\gamma+3\theta-6}}(\F_{\eps}(t))} \\
\leq & \, \eps \|\sqrt{\phi} (1-\eta_{\eps}) U_\eps : \nabla   \calB_{\Omega_2}[\psi \rho^{\theta}_{\eps}- \langle \psi \rho^{\theta}_{\eps} \rangle ] \|_{L^{\frac{6(\gamma+\theta)}{5(\gamma+\theta)-6}}(\F_{\eps}(t))} \\
\leq & \, \eps|\omega_{\eps}|\eps^{3\frac{5\gamma-\theta-6}{6(\gamma+\theta)}}\| \sqrt{\phi}\nabla   \calB_{\Omega_2}[\psi \rho^{\theta}_{\eps}- \langle \psi \rho^{\theta}_{\eps} \rangle ]\|_{L^{\frac{\gamma+\theta}{\theta}}(\F_{\eps}(t))} \\
\leq & \, \eps|\omega_{\eps}|\eps^{\frac{5\gamma-\theta-6}{2(\gamma+\theta)}}\| \sqrt{\phi} \psi \rho^{\theta}_{\eps} \|_{L^{\frac{\gamma+\theta}{\theta}}(\F_{\eps}(t))},
\end{align*}
where we use
\begin{gather*}
\frac{5\gamma+5\theta-6}{6(\gamma+\theta)} =\frac{5\gamma-\theta-6}{6(\gamma+\theta)} + \frac{\theta}{\gamma+\theta}.
\end{gather*}
We deduce
\begin{align*}
|B_{10}| \leq C  \eps |\omega_{\eps}|_{L^{\infty(0,T)}} \eps^{\frac{5\gamma-\theta-6}{2(\gamma+\theta)}} \| \phi^{1/{2\theta}} \psi^{1/\theta}\rho_{\eps}\|_{L^{\gamma+\theta}(0,T;L^{\gamma+\theta}(\F_{\eps}(t)))}^{\theta}.
\end{align*}

Note that the estimates for $ B_6 $, $ B_4 $, $ n_3^1 $ holds for $ \theta \leq 2\gamma/3 -1 $. It remains to show $ B_{5} $, $ n_3^3$, $ B_2 $ and $ B_1 $. Note that the difficulty to deal with $ B_{5} $, $ B_2 $ and $ B_1 $ is the presence of the term of the type $ \rho_{\eps}u_{\eps}\calB_{\Omega_2}[\rho^{\theta}_{\eps} \DIV(u_{\eps}) ]$.  
To estimate this term we use an interpolation inequality. Note that
\begin{align*}
\|\rho_{\eps}u_{\eps}\|_{L^r(0,T;L^s(\F_{\eps}(t)))} \leq C \|\rho_{\eps}\|_{L^{\gamma+\theta}(0,T;L^{\gamma+\theta}(\F_{\eps}(t)))}^{(1+\alpha)/2}
\end{align*}
with 
\begin{equation*}
\frac{1}{r} = \alpha \frac{\gamma+\theta+2}{2\gamma+2\theta}+(1-\alpha)\frac{1}{2\gamma+2\theta} \quad \text{ and } \quad \frac{1}{s} = \alpha\frac{\gamma+\theta+6}{6\gamma+6\theta}+(1-\alpha)\frac{\gamma+\theta+1}{2\gamma+2\theta}.
\end{equation*}
The dual exponent are
\begin{equation*}
\frac{1}{r'} = \alpha \frac{\gamma+\theta-2}{2\gamma+2\theta}+(1-\alpha)\frac{2\gamma+2\theta-1}{2\gamma+2\theta} \quad \text{ and } \quad \frac{1}{s'} = \alpha\frac{5\gamma+5\theta-6}{6\gamma+6\theta}+(1-\alpha)\frac{\gamma+\theta-1}{2\gamma+2\theta}.
\end{equation*}
\begin{align*}
\left| \int_0^t \int_{\F_{\eps}(t)}\sqrt{\phi} \psi \rho_{\eps} u_{\eps} \calB_{\Omega_2}[\sqrt{\phi}\psi\rho^{\theta}\DIV(u_{\eps})]\right| \leq & \, \|\sqrt{\phi} \psi \rho_{\eps} u_{\eps} \|_{L^r(0,T;L^s(\F_{\eps}(t)))}\|\calB_{\Omega_2}[\sqrt{\phi}\psi\rho^{\theta}\DIV(u_{\eps})]\|_{L^{r'}(0,T;L^{s'}(\F_{\eps}(t)))}.
\end{align*}
Moreover 
\begin{align*}
\|\calB_{\Omega_2}[\sqrt{\phi}\psi\rho^{\theta}\DIV(u_{\eps})]\|_{L^{r'}(0,T;L^{s'}(\F_{\eps}(t)))} \leq & \,\| \sqrt{\phi}\psi\rho^{\theta}\DIV(u_{\eps}) \|_{L^{r'}(0,T;L^{3s'/{3+s'}}(\F_{\eps}(t)))} \\
\leq & \, \|\DIV(u_{\eps})\|_{L^{2}(0,T;L^{2}(\F_{\eps}(t)))}\| \sqrt{\phi}\psi\rho^{\theta}\DIV(u_{\eps}) \|_{L^{2r'/{2-r'}}(0,T;L^{6s'/{6-s'}}(\F_{\eps}(t)))}
\end{align*}
To close the estimate we need to show that there exist $ \alpha \in [0,1] $ such that 
\begin{equation*}
\frac{2r'}{2-r'} \leq \frac{\gamma+\theta}{\theta} \quad \text{ and } \quad \frac{6s'}{6-s'} \leq \frac{\gamma+\theta}{\theta}.
\end{equation*}
Equivalently
\begin{equation*}
\frac{3\theta+\gamma}{2\gamma+2\theta} \leq \frac{1}{r'} \quad \text{ and } \quad \frac{7\theta+\gamma}{6\gamma+6\theta} \leq \frac{1}{s'}
\end{equation*}
After a small computation we have
{
\begin{equation*}
\frac{4\theta-2\gamma + 3}{2\gamma+2\theta -3} \leq \alpha \leq \frac{\gamma-\theta-1}{\gamma+\theta+1}.
\end{equation*}
Note that 
\begin{equation*}
\frac{4\theta-2\gamma + 3}{2\gamma+2\theta -3} \leq \frac{1}{5} \leq \frac{\gamma-\theta-1}{\gamma+\theta+1}.
\end{equation*}
for $ \theta \leq 2\gamma/3 -1$, in particular we can estimates $ B_5 $, $ B_2 $ and $ B_1 $.
}
%\footnote{I don't see how you got it.}
{ Let conclude with the term $ n_3^3$ we apply again the cut-off $\chi_{\eps}$ see \eqref{cuttoff}.}

\begin{align*}
\|\phi n_{3}^3 \|_{L^{2}(0,T;L^{6\gamma/(5\gamma-6)}(\F_{\eps(t)}))} \leq  & \, \| \phi \psi  \chi_{\eps}\rho_{\eps}^{\theta}(u_{\eps}-u_{s,\eps}) \|_{L^{2}(0,T;L^{6\gamma/(5\gamma-6)}(\F_{\eps(t)}))} \\
\leq & \, \|u_{\eps}\|_{_{L^{2}(0,T;L^{6}(\F_{\eps}(t)))}}\|\rho_{\eps}\|_{L^{\infty}(0,T;L^{\gamma}(\F_{\eps}(t)))} \\ & + (\|\ell_{\eps}\|_{L^{\infty}(0,T)}+\eps\|\omega_{\eps}\|_{L^{\infty}(0,T)}\|\chi_{\eps}\|_{L^{\infty}(0,T;L^{\frac{6\gamma(\gamma+\theta)}{5\gamma(\gamma+\theta)-6\gamma\theta}}(\F_{\eps}(t)))}\| \phi \psi \rho^{\theta}\|_{L^{2}(0,T;L^{(\gamma-\theta)/{\theta}}(\F_{\eps}(t)))}\\
\leq & \, \|u_{\eps}\|_{_{L^{2}(0,T;L^{6}(\F_{\eps}(t)))}}\|\rho_{\eps}\|_{L^{\infty}(0,T;L^{\gamma}(\F_{\eps}(t)))} \\
 & \, + (\|\ell_{\eps}\|_{L^{\infty}(0,T)}+\eps\|\omega_{\eps}\|_{L^{\infty}(0,T)})\eps^{\frac{5\gamma-\theta-6}{2(\gamma+\theta)}}\|\phi \psi \rho_{\eps}^{\theta}\|_{L^{2}(0,T;L^{(\gamma+\theta)/\theta}(\F_{\eps}(t)))},
\end{align*}
where we choose $ \|\chi_{\eps}\|_{L^{\frac{6\gamma(\gamma+\theta)}{5\gamma(\gamma+\theta)-6\gamma\theta}}(\F_{\eps}(t)))} = \eps^{\frac{5\gamma-\theta-6}{2(\gamma+\theta)}} $.

We show that for $ \theta \leq 2\gamma/3-1  $, the density $ \rho_{\eps} $ in uniformly bounded in $ L^{\gamma + \theta }$ provide that 
\begin{equation*}
%\label{ineq:fine}
|h_{\eps}'(t)|  \eps^{\frac{5\gamma-\theta-6}{2(\gamma+\theta)}} \leq C \quad \text{ and }  \quad \eps|\omega_{\eps}(t)| \eps^{\frac{5\gamma-\theta-6}{2(\gamma+\theta)}} \leq C, 
\end{equation*}
which hold true, combining the energy estimates with the assumptions \eqref{extra:hyp}. The proof is then finish.

\end{proof}

\subsection{An appropriate cut-off}
\label{sec:5}

We are now ready to pass to the limit in the weak formulation. The tricky term is the one involving the pressure, in fact it is not enough to consider a cut-off and its $\eps$-scales. The idea is to use cut-off that minimize in some sense the $ L^3 $ norm of the gradient. These types of cut-off have been widely used to treat these kind of problems. We recall the main properties and we refer to \cite{Je:2} for the proofs.

First of all for $ A, B \in \R $ with $ 0 < A < B $, we denote by $ \alpha = B/ A > 1 $ and we define the function
\begin{equation*}
f_{A,B}(z) = \begin{cases}
1 \quad & \text{ for } 0\leq z < A, \\
\frac{\log z - \log B }{\log A - \log B } \quad & \text{ for } A \leq z \leq B, \\
0 \quad & \text{ for } z > B.
\end{cases}
\end{equation*} 
It holds that $ f_{A,B} \in W^{1,\infty}(\R^{+}) $. We define the three dimensional cut-off 
\begin{equation*}
1- \eta_{\eps, \alpha_{\eps}}(x) = f_{\eps, \alpha_{\eps} \eps }(|x|),  
\end{equation*} 
where $ \alpha_{\eps} $ will be choose appropriately.

\begin{Proposition}
Under the hypothesis that $ \eps \alpha_{\eps} \to 0$, it holds

\begin{enumerate}

\item The functions $ 1- \eta_{\eps, \alpha_{\eps}} \longrightarrow 0 $  in $ L^{q}(\R^3) $ for $ 1 \leq q < +\infty $.

\item We have $$ \left\| \nabla \eta_{\eps, \alpha_{\eps}} \right\|_{L^{3}(\R^3)}^3 = \frac{2\pi^2}{(\log \alpha_{\eps})^2 }. $$

\item For $ 1\leq q < 3 $, $$ \left\| \nabla \eta_{\eps, \alpha_{\eps}} \right\|_{L^{q}(\R^3)}^q \leq 2\pi^2 \frac{\alpha_{\eps}^{3-q}}{(\log \alpha_{\eps})^{q-1} } \eps^{3-q}. $$

\end{enumerate}

\end{Proposition}

\begin{proof}
After passing to spherical coordinates the proof is straight-forward. See Lemma 2 of \cite{Je:2}.
\end{proof}

Let us fix $ \alpha_{\eps} $ in dependence of $ m_{\eps} $ and $ \calJ_{\eps} $ in such a way that 
\begin{equation}
\label{hyp:alpha:eps}
\eps \alpha_{\eps} \longrightarrow 0, \quad \alpha_{\eps} \longrightarrow + \infty, \quad  \lim_{\eps \to 0 } \frac{m_{\eps}}{(\eps \alpha_{\eps})^{\frac{3 \gamma-4}{\gamma}}}  = + \infty \quad \text{ and } \quad \lim_{\eps \to 0 } \inf_{\xi \in S^2}\frac{\xi \cdot \calJ_{\calS_{0,\eps}} \xi}{(\eps \alpha_{\eps})^{2+\frac{3 \gamma-4}{\gamma}}} = + \infty.
\end{equation} 
  
From now we write $ \eta_{\eps} $ instead of $ \eta_{\eps,\alpha_{\eps}} $ for a chosen sequence $ \alpha_{\eps} $ that satisfy \eqref{hyp:alpha:eps}.

\subsection{Pass to the limit in the weak formulation}
\label{Sec:4}

From the energy estimates \eqref{en:est} and the pressure estimate from Proposition \ref{PROP:2}, we deduce that
\begin{gather*}
\rho_{\eps} \cv \rho \quad \text{ in } \quad  L_{loc}^{3\gamma/2}\left([0,T); L_{loc}^{3\gamma/2}(\Omega)\right), \\
\rho_{\eps} \longrightarrow \rho \quad \text{ in } \quad  C_{w}(0,T; L^{\gamma}(\Omega)), \\ 
u_{\eps} \cv u \quad \text{ in } \quad L^{2}(0,T;H^1_0(\Omega)), \\
\rho_{\eps}u_{\eps} \to \rho u \quad \text{ in } \quad C_{w}(0,T; L^{2\gamma/(\gamma+1)}(\Omega)), \\
\eta_{\eps}\rho_{\eps}u_{\eps}\otimes u_{\eps} \longrightarrow \rho u\otimes u \quad \text{ in } \quad \mathcal{D}'([0,T)\times\Omega), \\
\rho_{\eps}^{\gamma} \longrightarrow \overline{\rho^{\gamma}} \quad \text{ in } \quad L^{3/2}_{loc}\left([0,T);L_{loc}^{3/2}(\Omega)\right)
\end{gather*}
where we use the momentum equation to show the second-last convergence. Recall that the weak formulation for the transport equation reads
\begin{equation*}
\int_{\F_{\eps}(0)} \rho_{0,\eps}\varphi(0,.) +  \int_{0}^{T}\int_{\F_{\eps}(t)} \rho_{\eps}\partial_t  \varphi + \rho_{\eps}u_{\eps}\cdot \nabla \varphi = 0,
\end{equation*}
for any $ \varphi \in C^{\infty}_{c}([0,T)\times \Omega)$. Passing to the limit with $ \eps $, we have 
\begin{equation*}
\int_{\F(0)} \rho_{0}\varphi(0,.) +  \int_{0}^{T}\int_{\F(t)} \rho_{\eps}\partial_t  \varphi + \rho u \cdot \nabla \varphi = 0.
\end{equation*}
We now pass to the limit in the momentum equation. For $ \varphi \in C^{\infty}_{c}([0,T)\times \Omega; \R^3)$ we test the weak formulation of the momentum equation with $ \varphi \eta_{\eps} $. It reads 

\begin{align*}
\int_{\Omega} q_{0,\eps} \eta(0,.) \varphi(0,.) + \int_{0}^{T} \int_{\Omega} (\rho_{\eps} u_{\eps}) \cdot \partial_t (\eta_{\eps} \varphi) + [\rho_{\eps} u_{\eps} \otimes u_{\eps}]: D(\eta_{\eps} \varphi) + \rho_{\eps}^{\gamma} \DIV( \eta_{\eps} \varphi) = \int_{0}^{T} \int_{\Omega} \bbS u_{\eps} : D (\eta_{\eps} \varphi).
\end{align*}
Let $ \eps $ goes to zero. We obtain
\begin{align*}
\int_{\Omega} & \, q_{0} \varphi(0,.) + \int_{0}^{T} \int_{\Omega} \rho u \cdot \partial_t \varphi + [\rho u \otimes u]: D \varphi + \overline{\rho^{\gamma}} \DIV( \varphi) - \int_{0}^{T} \int_{\Omega} \bbS u : D  \varphi \\
= & \, \lim_{\eps \to 0}  \Bigg(
- \int_{0}^{T} \int_{\Omega} (\rho_{\eps} u_{\eps}) \cdot\varphi u_{S,\eps} \cdot \nabla \eta_{\eps}  + [\rho_{\eps} u_{\eps} \otimes u_{\eps}]: \frac{1}{2}(\nabla \eta_{\eps} \otimes \varphi+ \varphi \otimes \nabla \eta_{\eps}) + \rho_{\eps}^{\gamma} \nabla \eta_{\eps} \cdot \varphi \\ & \quad \quad \quad + \int_{0}^{T} \int_{\Omega} \bbS u_{\eps} :  \frac{1}{2}(\nabla \eta_{\eps} \otimes \varphi+ \varphi \otimes \nabla \eta_{\eps})\Bigg).
\end{align*}
It remains to show that the right hand side is zero. To do that we show that any of the term converge to zero. 
\begin{align*}
\left| \int_{0}^{T} \int_{\Omega} (\rho_{\eps} u_{\eps}) \cdot\varphi u_{S,\eps} \cdot \nabla \eta_{\eps} \right| \leq & \, \|\rho_{\eps}\|_{L^{3\gamma/2}(\text{supp}(\varphi))}\| u_{\eps}\|_{L^2(0,T;L^{6}(\Omega))}|u_{S,\eps}|_{L^{\infty}(0,T)}\|\nabla\eta\|_{L^{\infty}(0,T;L^{6\gamma/(5\gamma-4)})} \\
\leq & \, C |u_{S,\eps}|_{L^{\infty}(0,T)} \frac{(\eps\alpha_{\eps})^{\frac{3\gamma-4}{2\gamma}}}{(\log \alpha_{\eps})^\frac{\gamma+4}{6\gamma}} \longrightarrow 0.
\end{align*}
The second term reads
\begin{align*}
\left| \int_{0}^{T} \int_{\Omega} [\rho_{\eps} u_{\eps} \otimes u_{\eps}]: \frac{1}{2}(\nabla \eta_{\eps} \otimes \varphi+ \varphi \otimes \nabla \eta_{\eps})  \right| \leq & \,  \|\rho_{\eps}\|_{L^{\infty}(0,T;L^{\gamma}(\Omega))}\|u_{\eps}\|_{L^2(0,T;L^6(\Omega)}^2\|\nabla \eta_{\eps} \|_{L^{\infty}(0,T;L^{3}(\Omega))} \\
\\ \leq & \, \frac{C}{(\log \alpha_{\eps})^{2/3}} \longrightarrow 0.
\end{align*}
where we use $ \gamma \geq 3 $.
Similarly
\begin{align*}
\left| \int_{0}^{T} \int_{\Omega} \rho_{\eps}^{\gamma} \nabla \eta_{\eps} \cdot \varphi \right| \leq \|\rho_{\eps}^{\gamma}\|_{L^{3/2}(\text{supp}(\varphi))}\|\nabla \eta_{\eps}\|_{L^{\infty}(0,T;L^{3}(\Omega)} \leq  \frac{C}{(\log \alpha_{\eps})^{2/3}} \longrightarrow 0.
\end{align*}
Finally 
\begin{align*}
\left| \int_{0}^{T} \int_{\Omega} \bbS u_{\eps} :  \frac{1}{2}(\nabla \eta_{\eps} \otimes \varphi+ \varphi \otimes \nabla \eta_{\eps}) \right| \leq & \, \|u_{\eps}\|_{L^2(0,T;W^{1,2}(\Omega)}\|\nabla \eta_{\eps} \|_{L^{\infty}(0,T;L^2(\Omega))} \\
\leq & \,  C\sqrt{\frac{\eps \alpha_{\eps}}{\log \alpha_{\eps}} }  \longrightarrow 0.
\end{align*}

To conclude is enough to identify the limit of the pressure, more precisely to show that
\begin{equation*}
\overline{\rho^{\gamma}} = \rho^{\gamma}.
\end{equation*}
It is now well-known how to proceed in this final step and we briefly present the proof in the next section.

\subsection{Identification of the pressure}

A key tool to identify the pressure is the study of the so-called effective viscous flux which enjoys a better compactness property. In what follow we prove the key lemma that is needed to follow the classical proofs presented in \cite{Lion:2} or \cite{NovS}.  

\begin{Lemma}

For any $ \psi \in C^{\infty}_c(\Omega) $, it holds
\begin{equation*}
\lim_{\eps\to 0} \int_{0}^T \int_{\Omega} \psi \eta_{\eps}\left(\rho_{\eps}^{\gamma}-(2\mu + \lambda)\DIV(u_{\eps})\right)\rho_{\eps} = \int_{0}^T \int_{\Omega} \psi \left(  \overline{\rho^{\gamma}}-(2\mu + \lambda)\DIV(u)\right)\rho 
\end{equation*}
up to subsequence.

\end{Lemma}

\begin{proof}

Consider $ \phi_{\eps} = \psi \eta_{\eps} \nabla \Delta^{-1}[\rho_{\eps}] $ and $ \phi = \psi \nabla \Delta^{-1}[\rho] $. From the energy estimate  $ \psi \nabla \Delta^{-1}[\rho_{\eps}] $ is uniformly bounded in $ L^{\infty}(0,T;W^{1,\gamma}(\Omega)) $
and
\begin{equation*}
\partial_t \psi \nabla \Delta^{-1}[\rho_{\eps}] =  \psi \nabla \Delta^{-1}[\DIV(\rho_{\eps}u_{\eps})]
\end{equation*}
is uniformly bounded in some $ L^p $. We deduce that
\begin{equation}
\label{cwcon}
\psi \nabla \Delta^{-1}[\rho_{\eps}] \longrightarrow \psi \nabla \Delta^{-1}[\rho] \quad \text{ in }  C_{w}(0,T;W^{1,\gamma}(\Omega)) 
\end{equation}
and from the fact that the compact embedding of $ W^{1,\gamma}(\Omega) \subset C^{0}(\Omega) $ the convergence is strong in $ C^{0}((0,T)\times \Omega)$. Using \eqref{cwcon} and the convergence of the initial data we have
\begin{align*}
\lim_{\eps \to 0} & \, \int_{0}^{T} \int_{\Omega} \rho_{\eps} u_{\eps} \cdot \partial_t  \phi_{\eps} + [\rho_{\eps} u_{\eps} \otimes u_{\eps}]: \nabla \phi_{\eps} + \rho_{\eps}^{\gamma} \DIV( \phi_{\eps}) -  \mu \nabla u_{\eps} : \nabla \phi_{\eps}-  (\mu+\lambda) \DIV( u_{\eps}) : \DIV(\phi_{\eps}) \\
= & \, \int_{0}^{T} \int_{\Omega} \rho u \cdot \partial_t  \phi + [\rho u \otimes u]: \nabla \phi + \overline{\rho^{\gamma}} \DIV( \phi) -  \mu \nabla u : \nabla \phi -  (\mu+\lambda) \DIV( u) : \DIV(\phi).
\end{align*} 
Using the definition of $ \phi_{\eps} $ and $ \phi $ we rewrite the above equality as follows.
\begin{gather*}
\lim_{\eps \to 0}  \int_{0}^{T} \int_{\Omega} \eta_{\eps}\psi\left( \rho_{\eps}^{\gamma} \rho_{\eps} - \mu \nabla u_{\eps}: \nabla^2\Delta^{-1}[\rho_{\eps}]-(\mu+\lambda)\DIV(u_{\eps})\rho^{\eps}  \right) \\
- \int_{0}^{T} \int_{\Omega} \psi\left( \overline{\rho^{\gamma}} \rho - \mu \nabla u: \nabla^2\Delta^{-1}[\rho]-(\mu+\lambda)\DIV(u)\rho \right) \\
= - \lim_{\eps \to 0}  \int_{0}^{T} \int_{\Omega} \eta_{\eps}\rho_{\eps}^{\gamma}\nabla\psi \cdot \nabla \Delta^{-1}[\rho_{\eps}] + \int_{0}^{T} \int_{\Omega} \overline{\rho^{\gamma}} \nabla \psi \cdot \nabla \Delta^{-1}[\rho] \\
- \lim_{\eps \to 0}  \int_{0}^{T} \int_{\Omega} \psi \rho_{\eps}^{\gamma}\nabla\eta_{\eps} \cdot \nabla \Delta^{-1}[\rho_{\eps}] \\
+ \mu \lim_{\eps \to 0}  \int_{0}^{T} \int_{\Omega} \eta_{\eps} \nabla u_{\eps}: (\nabla \psi \otimes \nabla \Delta^{-1}[\rho_{\eps}]) - \mu \int_{0}^{T} \int_{\Omega} \nabla u : (\nabla \psi \otimes \nabla \Delta^{-1}[\rho])  \\
+  \mu \lim_{\eps \to 0}  \int_{0}^{T} \int_{\Omega} \psi \nabla u_{\eps}: (\nabla \eta_{\eps} \otimes \nabla \Delta^{-1}[\rho_{\eps}]) \\
+ (\mu+\lambda) \lim_{\eps \to 0}  \int_{0}^{T} \int_{\Omega} \eta_{\eps} \DIV (u_{\eps}): \nabla \psi \cdot \nabla \Delta^{-1}[\rho_{\eps}]-(\mu+\lambda)  \int_{0}^{T} \int_{\Omega}  \DIV (u): \nabla \psi \cdot \nabla \Delta^{-1}[\rho] \\
(\mu+\lambda) \lim_{\eps \to 0}  \int_{0}^{T} \int_{\Omega} \psi \DIV (u_{\eps}): \nabla \eta_{\eps} \cdot \nabla \Delta^{-1}[\rho_{\eps}] \\
- \lim_{\eps \to 0}  \int_{0}^{T} \int_{\Omega} \psi \rho_{\eps} u_{\eps}\otimes u_{\eps}:(\nabla \eta_{\eps}\otimes \nabla \Delta^{-1}[\rho_{\eps}]) - \lim_{\eps \to 0}  \int_{0}^{T} \int_{\Omega} \psi \rho_{\eps}u_{\eps}\cdot \nabla \Delta^{-1}[\rho_{\eps}] u_{S,\eps} \nabla \eta_{\eps} \\
- \lim_{\eps \to 0}  \int_{0}^{T} \int_{\Omega} \eta_{\eps} \rho_{\eps} u_{\eps}\otimes u_{\eps}:(\nabla \psi \otimes \nabla \Delta^{-1}[\rho_{\eps}]) + \int_{0}^{T} \int_{\Omega}  \rho u \otimes u :(\nabla \psi \otimes \nabla \Delta^{-1}[\rho]) \\
- \lim_{\eps \to 0}  \int_{0}^{T} \int_{\Omega} \psi \eta_{\eps} u_{\eps} \cdot \left[\rho_{\eps} (u_{\eps}\cdot \nabla) \nabla \Delta^{-1}[\rho_{\eps}])- \rho_{\eps}\nabla\Delta^{-1}[\DIV(\rho_{\eps}u_{\eps})] \right] \\
-  \int_{0}^{T} \int_{\Omega} \psi u \cdot \left[\rho (u \cdot \nabla) \nabla \Delta^{-1}[\rho])- \rho \nabla\Delta^{-1}[\DIV(\rho u )] \right].
\end{gather*}
First of all notice that any line of the right hand side except the last two are zero by using the convergence in $ C^{0}((0,T)\times \Omega) $ of $ \nabla \Delta^{-1}[\rho_{\eps}] $ and the convergences presented at the beginning of Section \ref{Sec:4}. Regarding the last two lines we recall that applying the Div-Curl Lemma we deduce
\begin{gather*}
\rho_{\eps} (u_{\eps}\cdot \nabla) \nabla \Delta^{-1}[\rho_{\eps}])- \rho_{\eps}\nabla\Delta^{-1}[\DIV(\rho_{\eps}u_{\eps})] \longrightarrow \rho (u \cdot \nabla) \nabla \Delta^{-1}[\rho])- \rho \nabla\Delta^{-1}[\DIV(\rho u )] \quad \text{ in } C_{w}(0,T;L^\frac{2\gamma}{\gamma+3}(\Omega))
\end{gather*} 
in particular the convergence is strong in $ L^{q}(0,T; (W^{1,p}(\Omega))^{*})$ for any $ q $ and $ p > 6\gamma/(5\gamma-9) $, where we use the compactness of the Sobolev embeddings for the exponents 
\begin{equation*}
1-\frac{3}{p} > -3\left(1-\frac{\gamma+3}{2\gamma} \right).
\end{equation*} 
Note that $ \eta_{\eps} u_{\eps} \cv u $ in $ L^{2}(0,T;W^{1,p}(\Omega)) $ for $ p \leq 2 $. Finally the convergence follows because $ 6\gamma/(5\gamma-9) < 2 $ for $ \gamma \geq 6 $.

Regarding the left hand side we rewrite the term
\begin{align*}
\int_{0}^{T} \int_{\Omega} \eta_{\eps}\psi \nabla u_{\eps}: \nabla^2\Delta^{-1}[\rho_{\eps}] = & \, \sum_{i,j} \int_{0}^{T} \int_{\Omega} \partial_i (\eta_{\eps}\psi  u_{\eps,j})\partial_i \partial_j \Delta^{-1}[\rho_{\eps}] - \sum_{i,j} \int_{0}^{T} \int_{\Omega}  \eta_{\eps}\partial_i\psi  u_{\eps,j}\partial_i \partial_j \Delta^{-1}[\rho_{\eps}] \\ & \, - \sum_{i,j} \int_{0}^{T} \int_{\Omega} \partial_i \eta_{\eps}\psi  u_{\eps,j}\partial_i \partial_j \nabla^2\Delta^{-1}[\rho_{\eps}] \\
= & \, \int_{0}^{T} \int_{\Omega}  \eta_{\eps}\psi  \DIV(u_{\eps}) \rho_{\eps} + \int_{0}^{T} \int_{\Omega}  \psi \nabla \eta_{\eps} \cdot \DIV(u_{\eps}) \rho_{\eps}+ \int_{0}^{T} \int_{\Omega}  \eta_{\eps}\nabla\psi \cdot  u_{\eps} \rho_{\eps} \\ &\, - \sum_{i,j} \int_{0}^{T} \int_{\Omega}  \eta_{\eps}\partial_i\psi  u_{\eps,j}\partial_i \partial_j \Delta^{-1}[\rho_{\eps}] - \sum_{i,j} \int_{0}^{T} \int_{\Omega} \partial_i \eta_{\eps}\psi  u_{\eps,j}\partial_i \partial_j \nabla^2\Delta^{-1}[\rho_{\eps}].
\end{align*}
Similarly
\begin{align*}
\int_{0}^{T} \int_{\Omega} \psi \nabla u: \nabla^2\Delta^{-1}[\rho] = 
 \int_{0}^{T} \int_{\Omega}  \psi  \DIV(u) \rho + \int_{0}^{T} \int_{\Omega}  \nabla \psi \cdot  u \rho - \sum_{i,j} \int_{0}^{T} \int_{\Omega}  \partial_i\psi  u_{j} \partial_i \partial_j \Delta^{-1}[\rho].
\end{align*}
We deduce that 
\begin{gather*}
 \int_{\Omega} \psi \eta_{\eps}\left(\rho_{\eps}^{\gamma}-(2\mu + \lambda)\DIV(u_{\eps})\right)\rho_{\eps} - \int_{\Omega} \psi \left(  \overline{\rho^{\gamma}}-(2\mu + \lambda)\DIV(u)\right)\rho  \\
= \mu \int_{0}^{T} \int_{\Omega}  \psi \nabla \eta_{\eps} \cdot \DIV(u_{\eps}) \rho_{\eps} - \mu  \sum_{i,j} \int_{0}^{T} \int_{\Omega} \partial_i \eta_{\eps}\psi  u_{\eps,j}\partial_i \partial_j \nabla^2\Delta^{-1}[\rho_{\eps}]  \\ - \mu  \sum_{i,j} \int_{0}^{T} \int_{\Omega}  \eta_{\eps}\partial_i\psi  u_{\eps,j}\partial_i \partial_j \Delta^{-1}[\rho_{\eps}] + \mu  \sum_{i,j}  \int_{0}^{T} \int_{\Omega}  \partial_i\psi  u_{j} \partial_i \partial_j \Delta^{-1}[\rho]
\\ + \mu  \int_{0}^{T} \int_{\Omega}  \eta_{\eps}\nabla\psi \cdot  u_{\eps} \rho_{\eps} - \mu  \int_{0}^{T} \int_{\Omega}  \nabla \psi \cdot  u \rho.
\end{gather*}
As before the right hand side converges to zero and the lemma is  proved.

\end{proof}

\begin{Remark}
Note that we are in the case $ \gamma \geq 6 $ so to conclude it is enough to follow Lions approach
\end{Remark}
\bigskip

{\bf Acknowledgements.} 

M.B. is supported by the ERCEA under the grant 014 669689-HADE and also by the Basque Government through the BERC 2014-2017 program and by Spanish Ministry of Economy and Competitiveness MINECO: BCAM Severo Ochoa excellence accreditation SEV-2013-0323. M.B. warmly thank Prof. \v{S}\'arka Ne\v{c}asov\'a and the Institute of Mathematics of the Czech Academy of Sciences for the kind hospitality in October 2019. 
S.N.  has been supported by the Czech Science Foundation (GA\v CR) project GA19-04243S. The Institute of Mathematics, CAS is supported by RVO:67985840.


\begin{thebibliography}{50}

\bibitem{All} Allaire, G. (1991). Homogenization of the Navier-Stokes equations in open sets perforated with tiny holes I. Abstract framework, a volume distribution of holes. Archive for Rational Mechanics and Analysis, 113(3), 209-259.

\bibitem{All:2} Allaire, G. (1991). Homogenization of the Navier-Stokes equations in open sets perforated with tiny holes II: Non-critical sizes of the holes for a volume distribution and a surface distribution of holes. Archive for rational mechanics and analysis, 113(3), 261-298.

\bibitem{MB1} Bravin, M. (2019). Energy Equality and Uniqueness of Weak Solutions of a ``Viscous Incompressible Fluid+ Rigid Body'' System with Navier Slip-with-Friction Conditions in a 2D Bounded Domain. Journal of Mathematical Fluid Mechanics, 21(2), 23.

 \bibitem{DE}
Desjardins, B. and  M. J. Esteban. (2000)
 On weak solutions for fluid-rigid structure interaction: Compressible and 
incompressible models.
{ Commun. Partial Differential Equations} {25}, 1399--1413.

\bibitem{DFL} Diening, L., Feireisl, E., Lu, Y. (2017). The inverse of the divergence operator on perforated domains with applications to homogenization problems for the compressible Navier-Stokes system. ESAIM Control Optim. Calc. Var. 23  (3), 851--868.

\bibitem{F:body:comp} Feireisl, E. (2003). On the motion of rigid bodies in a viscous compressible fluid. Archive for Rational Mechanics and Analysis, 167(4), 281.

\bibitem{Fei:Lu} Feireisl, E., Lu, Y. (2015). Homogenization of stationary Navier-Stokes equations in domains with tiny holes. Journal of Mathematical Fluid Mechanics, 17(2), 381-392.

\bibitem{FT} Feireisl, E., Novotn\'y, A., Takahashi, T. (2010). Homogenization and singular limits for the complete Navier–Stokes–Fourier system. Journal de Mathématiques Pures et Appliquées, 94(1), 33-57.

\bibitem{FNN}  Feireisl, E., Namlyeyeva, Y., Ne\v casov\' a, \v S. (2016). Homogenization of the evolutionary Navier-Stokes system. Manuscripta Math. 149, (1-2), 251--274.

\bibitem{Gal} Galdi, G. (2011). An introduction to the mathematical theory of the Navier-Stokes equations: Steady-state problems. Springer Science \& Business Media.

\bibitem{GH:exi} G\'{e}rard-Varet, D., Hillairet, M. (2014). Existence of Weak Solutions Up to Collision for Viscous Fluid-Solid Systems with Slip. Communications on Pure and Applied Mathematics, 67(12), 2022-2076.

\bibitem{GLS} Glass, O., Lacave, C., Sueur, F. (2016). On the motion of a small light body immersed in a two dimensional incompressible perfect fluid with vorticity. Communications in Mathematical Physics, 341(3), 1015-1065.

\bibitem{G:S} Glass, O., Sueur, F. (2019). Dynamics of several rigid bodies in a two-dimensional ideal fluid and convergence to vortex systems. arXiv preprint arXiv:1910.03158.

\bibitem{Je:1} He, J., Iftimie, D. (2018). On the small rigid body limit in 3D incompressible flows. arXiv preprint arXiv:1812.09196.

\bibitem{Je:2} He, J., Iftimie, D. (2019). A small solid body with large density in a planar fluid is negligible. Journal of Dynamics and Differential Equations, 31(3), 1671-1688.


\bibitem{ILL} Iftimie, D., Lopes, M. C., Lopes, H. J. N. (2003). Two dimensional incompressible ideal flow around a small obstacle. Communications In Partial Differential Equations, 28(1-2), 349-379.

\bibitem{ILL:visc} Iftimie, D., Lopes Filho, M. C., Lopes, H. N. (2006). Two-dimensional incompressible viscous flow around a small obstacle. Mathematische Annalen, 336(2), 449.

\bibitem{KrNePi_2}
Kreml O., Ne\v casov\'a, \v S., Piasecki, T. (2020).
Weak-strong uniqueness for the compressible fluid-rigid body interaction.
{J. Differential Equations}, (2020);
{268}, 4756--4785.

\bibitem{KHS} Kowalczyk, K., and H\"ofer, R. M., Schwarzacher, S. (2020). Darcy's law as low Mach and homogenization limit of a compressible fluid in perforated domains. arXiv preprint arXiv:2007.09031. 


\bibitem{Lac} Lacave, C. (2009). Two-dimensional incompressible viscous flow around a thin obstacle tending to a curve. Proceedings of the Royal Society of Edinburgh Section A: Mathematics, 139(6), 1237-1254.

\bibitem{La:Ta} Lacave, C., Takahashi, T. (2017). Small moving rigid body into a viscous incompressible fluid. Archive for Rational Mechanics and Analysis, 223(3), 1307-1335.

\bibitem{Lion:2} Lions, P. L. (1996). Mathematical Topics in Fluid Mechanics: Volume 2: Compressible Models (Vol. 2). Oxford University Press on Demand.

\bibitem{Lu:Schw} Lu, Y., Schwarzacher, S. (2018). Homogenization of the compressible Navier–Stokes equations in domains with very tiny holes. Journal of Differential Equations, 265(4), 1371-1406.

\bibitem{LP} Lu, Y., Pokorn\' y, M. (2020). Homogenization of stationary Navier–Stokes–Fourier system in domains
with tiny holes. arXiv preprint arXiv:2001.06950

\bibitem{Lu} Lu, Y. (2018) Uniform estimates for Stokes equations in a domain with a small hole and applications in
homogenization problems. Preprint, arXiv:1510.01678.

\bibitem{Mas} Masmoudi, N. (2002). Homogenization of the compressible Navier–Stokes equations in a porous medium. ESAIM: Control, Optimisation and Calculus of Variations, 8, 885-906.

\bibitem{M} Mikeli\v c, A. (1991). Homogenization of nonstationary Navier-Stokes equations in a domain with a grained boundary. Annali di Matematica pura ed applicata, 158(1), 167-179.

\bibitem{NovS} Novotn\'y, A., Stra\v skraba, I. (2004). Introduction to the mathematical theory of compressible flow (No. 27). Oxford University Press on Demand.

\bibitem{Tar} Tartar, L. (1980). Incompressible fluid flow in a porous medium-convergence of the homogenization process. Appendix of Non-homogeneous media and vibration theory, edited by E. S\'anchez-Palencia, 368-377.

\end{thebibliography}
\end{document}